\documentclass{amsart}
\usepackage{hyperref}

\numberwithin{equation}{section}

\newtheorem{theorem}[equation]{Theorem}
\newtheorem{lemma}[equation]{Lemma}
\newtheorem{cor}[equation]{Corollary}

\newtheorem{conj}[equation]{Conjecture}
\theoremstyle{definition}
\newtheorem{remark}[equation]{Remark}

\newcommand{\CC}{\mathbb{C}}
\newcommand{\FF}{\mathbb{F}}

\newcommand{\QQ}{\mathbb{Q}}
\newcommand{\RR}{\mathbb{R}}
\newcommand{\ZZ}{\mathbb{Z}}

\DeclareMathOperator{\coker}{coker}

\DeclareMathOperator{\res}{res}
\DeclareMathOperator{\Trace}{Trace}

\newcommand{\SageMath}{\textsc{SageMath}}
\newcommand{\Magma}{\textsc{Magma}}
\newcommand{\avlink}[1]{\href{http://www.lmfdb.org/Variety/Abelian/Fq/#1}{\textsf{#1}}}

\newcommand{\arXiv}[3]{\href{https://arxiv.org/abs/#1}{arXiv:#1v#2} (#3)}

\begin{document}

\title{The relative class number one problem for function fields, I}
\author{Kiran S. Kedlaya}
\date{August 23, 2022}

\thanks{Thanks to Xander Faber for providing an early draft of \cite{faber-grantham-howe},
Thomas Grubb for bringing the work of Dragutinovi\'c to our attention, and Drew Sutherland for help with computing abelian extensions of function fields in \Magma. The author was supported by NSF (grants DMS-1802161, DMS-2053473) and UC San Diego (Warschawski Professorship).}

\begin{abstract}
We reduce the classification of finite extensions of function fields (of curves over finite fields) with the same class number to a finite computation;
complete this computation in all cases except when both curves have base field $\FF_2$ and genus $>1$; and give a conjectural answer in the remaining cases. The conjecture will be resolved in subsequent papers.
\end{abstract}

\maketitle

\section{Introduction}

The \emph{relative class number one problem} for function fields (of curves over finite fields) is to classify finite extensions for which the relative class number equals 1, or equivalently the class numbers of the two function fields coincide.
In this paper, we solve this problem in all cases except where both function fields have base field $\FF_2$, and to reduce that case to a \emph{feasible} finite computation. This extends work of numerous authors \cite{bae-kang, jung-ahn, kida-murayabashi, leitzel-madan, macrae} but our arguments are independent of these.

For comparison, the relative class number one problem for number fields was formulated by Stark
\cite{stark} only for \emph{CM fields}, viewed as totally imaginary quadratic extensions of totally real fields. This restriction is quite natural: outside of this case, the relative unit rank is nonzero and the relative class number behaves erratically
(e.g., it is not generally integral). Odlyzko \cite{odlyzko} established conditionally on GRH that there are only finitely many CM fields with relative class number one. The complete set of  \emph{normal} CM fields with relative class number one has been determined recently by Hoffman--Sircana \cite{hofmann-sircana}.

Before continuing, we introduce some terminology and notation.
By a \emph{function field}, we mean the field of rational functions on a curve over some finite field.
Given a finite extension $F'/F$ of function fields, we write
$C, C'$ for the curves corresponding to $F,F'$;
$q_F, q_{F'}$ for the orders of the base fields of $C,C'$;
$g_F, g_{F'}$ for the genera of $C, C'$;
and $h_F, h_{F'}$ for the class numbers of $F, F'$.
We write $J(C), J(C')$ for the Jacobians of $C,C'$, so that $\#J(C)(\FF_{q_F}) = h_F$ and $\#J(C')(\FF_{q_{F'}}) = h_{F'}$.

The \emph{relative class number} $h_{F'/F}$ is the ratio $h_{F'}/h_F$; this can be interpreted as the order of a certain finite group (see below), and hence is an integer. This implies the following reduction: for $E = F \cdot \FF_{q_{F'}}$, $h_{F'/F} = 1$ if and only if $h_{E/F} = h_{F'/E} = 1$.
We may thus focus on the cases where $F' = E$, in which case we say the extension $F'/F$ is \emph{constant},
and where $E = F$, in which case we say $F'/F$ is \emph{purely geometric}.

In the case of a constant extension, the equality $h_{F'/F} = 1$ holds for trivial reasons when $F' = F$ and when $g_F = g_{F'} = 0$ (as in this case $h_F = h_{F'} = 1$).
Excluding these, we have the following result; see \S\ref{sec:constant extensions} for the proof.
\begin{theorem} \label{T:constant bounds}
Let $F'/F$ be a constant extension of degree $d>1$ of function fields with $g_{F} > 0$, $q_{F'} > q_F $, and $h_{F'/F} = 1$. Then $(q_F, d, g_F, J(C))$ is one of
\begin{gather*}
(2,2,1,\mbox{\avlink{1.2.c}}), (2,2,2,\mbox{\avlink{2.2.c\_c}}),
(2,2,2,\mbox{\avlink{2.2.d\_f}}), (2,2,3,\mbox{\avlink{3.2.e\_j\_p}}), \\
(2,3,1,\mbox{\avlink{1.2.b}}), (2,3,1,\mbox{\avlink{1.2.c}}),
(3,2,1,\mbox{\avlink{1.3.ad}}), (4,2,1,\mbox{\avlink{1.4.ae}}),
\end{gather*}
where $J(C)$ is specified up to isogeny by an LMFDB label.
\end{theorem}

In the case of a purely geometric extension, the equality $h_{F'/F} = 1$ holds for trivial reasons when $F' = F$ and when $g_F = g_{F'}  \in \{0,1\}$.
Moreover, when $g_F \in \{0,1\}$, for any fixed pair of isomorphism classes of $F$ and $F'$, the existence of a single finite morphism $F \to F'$ implies the existence of infinitely many more. It is thus natural to separate the cases $g_F \leq 1$ and $g_F > 1$; see \S\ref{sec:exhaustion} and \S\ref{sec:geometric q34} for the proofs.
\begin{theorem} \label{T:purely geometric bounds1}
Let $F'/F$ be a purely geometric extension of degree $d$ of function fields with $g_F \leq 1$, $g_{F'} > g_F$, and $h_{F'/F} = 1$.
Then $(q_F,g_F, g_F',J(C),J(C'))$ appears in Table~\ref{table:geometric extensions big q1}. (Note that the tuple does not always uniquely determine $F'$.)
\end{theorem}
When $g_F = 0$, Theorem~\ref{T:purely geometric bounds1} recovers the solution of the \emph{absolute class number one problem} for function fields
\cite{leitzel-madan-queen, stirpe, mercuri-stirpe, shen-shi}.

\begin{theorem} \label{T:purely geometric bounds2}
Let $F'/F$ be a purely geometric extension of degree $d$ of function fields with $g_{F'} > g_F > 1$ and $h_{F'/F} = 1$.
\begin{enumerate}
\item[(a)]
If $q_F > 2$, then $q_F \in \{3,4\}$, $(g_F, g_F') \in \{(2,3), (2,4), (3,5)\}$, $F'/F$ is (Galois) cyclic, and $(q_F, g_F, g_{F'}, F)$ appears in Table~\ref{table:geometric extensions big q2}. In each listed case, the tuple uniquely determines $F'$.
\item[(b)]
If $q_F = 2$, then $g_F \leq 7$ and $g_{F'} \leq 13$. The isogeny classes of $J(C)$ and the Prym variety $A$ (see below) form one of $208$ pairs listed 
in Table~\ref{table:geometric bounds}.
\item[(c)]
If $q_F = 2$, then assuming that $F'/F$ is cyclic, there are exactly $61$ tuples $(d,g_F, g_{F'}, F)$ with $g_F \notin \{6,7\}$, and at least $3$ with $g_F \in \{6,7\}$; see Tables~\ref{table:geometric extensions big q2b1} and~\ref{table:geometric extensions big q2b}. In each listed case, the tuple uniquely determines $F'$.
\end{enumerate}
\end{theorem}

In Theorem~\ref{T:purely geometric bounds2}(c), there are only two cases (\avlink{3.2.ab\_a\_c} and \avlink{5.2.b\_c\_e\_i\_i}) where $F$ is not uniquely specified by $d, g_F, g_{F'}, J(C)$. The scarcity of such examples reflects that curves with isogenous Jacobians can typically be distinguished by the $L$-functions of their abelian covers \cite{booher-voloch}.

By our earlier reduction, we recover the following corollary.
\begin{cor}
Let $F'/F$ be an extension of degree $d$ of function fields with $g_{F'} > g_F$ and $h_{F'/F} = 1$ which is neither constant nor purely geometric. Then $q_F = 2$, $q_{F'} = 4$, and
$(g_F, g_{F'}, J(C), J(C')) \in \{(0,1,0, \avlink{1.4.ae}), (1,2,\avlink{1.2.c}, \avlink{2.4.ae\_i})\}$.
\end{cor}

We now summarize the techniques used to prove Theorem~\ref{T:constant bounds}, Theorem~\ref{T:purely geometric bounds1}, and Theorem~\ref{T:purely geometric bounds2}. 
The extension $F'/F$ induces an injective morphism $f$ from $J(C)$ to the Weil restriction of $J(C')$ from $\FF_{q_{F'}}$ to $\FF_{q_F}$,
and $h_{F'/F}$ can be interpreted as the order of the group $A(\FF_{q_F})$ where $A$ is the cokernel of $f$;
we call $A$ the \emph{Prym variety} of the covering $C' \to C$.
We restrict options for $C$ and $C'$ using the structure of simple abelian varieties of order $1$ over $\FF_q$: for $q \geq 5$ there are none;
for $q = 3,4$ there are only elliptic curves;
for $q=2$ there is an infinite series described in work of Madan--Pal \cite{madan-pal}
and Robinson \cite{robinson}.

The severe restrictions on $A$ impose constraints in turn on the number of rational points on $C$ and $C'$ over various finite extensions of their base fields. In the constant case, the restrictions lead quickly to Theorem~\ref{T:constant bounds} because the zeta function of $C'$ is uniquely determined by the zeta function of $C$ and the degree of the extension. By contrast, in the purely geometric case there is no obvious way to predict the zeta function of $C'$ from that of $C$; we instead argue that $C$ is forced to have many rational points, which for $g_F \gg 0$ will violate a ``linear programming'' bound \cite[Part II]{serre-rational}. 
This yields effective upper bounds on $g_F$ and $g_{F'}$;
we then obtain a list of candidates for the Weil polynomials of $F$ and $F'$
by an exhaustion in \SageMath{} (as described in \cite{kedlaya-root}, and later used in LMFDB as per \cite{lmfdb-av}). There is a loose parallel here with the Serre--Lauter method for refining upper bounds on rational points on curves over finite fields \cite{lauter}.

To complete the proofs, we identify candidates for $C$ with a given zeta function
using data from LMFDB \cite{lmfdb}, which includes a table of genus-4 curves by Xarles \cite{xarles}, plus a similar table of genus-5 curves computed by Dragutinovi\'c \cite{dragutinovic}.
We then make a computation of abelian extensions of function fields in \Magma{}.

The relative class number one problem is now reduced to the following.
\begin{conj} \label{conj:remaining problem}
Let $F'/F$ be a purely geometric extension of degree $d>1$ with $q_F = 2$, $g_F > 1$, and $h_{F'/F} = 1$.
Then $F$ appears in one of Tables~\ref{table:geometric extensions big q2b1} or~\ref{table:geometric extensions big q2b}
\end{conj}
By Theorem~\ref{T:purely geometric bounds2}, this further reduces to the following two logically independent statements,
which will be addressed in subsequent work \cite{part2, part3}.
\begin{itemize}
\item
Any extension as in Conjecture~\ref{conj:remaining problem} is cyclic. This will follow from Theorem~\ref{T:purely geometric bounds2}(b) by extending the argument for $q_F > 2$ (see Lemma~\ref{L:degree 3 cyclic}).
\item
Table~\ref{table:geometric extensions big q2b} is complete in genera 6 and 7.
This will follow from a limited census based on Mukai's descriptions of canonical curves of these genera \cite{mukai-cg, mukai}; the entries in Table~\ref{table:geometric extensions big q2b} come from a preliminary version of this census.
\end{itemize}

We have not considered the relative class number $m$ problem for $m > 1$, as in \cite{leitzel-madan-queen}. This would require adapting Lemma~\ref{lem:excess} to abelian varieties over $\FF_2$ of order $m$.
For each $m$ it is known that there are infinitely many simple abelian varieties of order $m$ over $\FF_2$ \cite{kedlaya-construction}, but it seems hopeless to give a complete classification;
a better approach might be modeled on the use of resultants to prove statements
about small algebraic integers (see \cite{smith} for recent progress in this direction).

All computations in \SageMath{} \cite{sage} and \Magma{} \cite{magma}
are documented in Jupyter notebooks available from a GitHub repository \cite{repo};
the computations take under 2 hours on a single CPU (Intel i5-1135G7@2.40GHz)
and generate an Excel spreadsheet of the 208 pairs of Weil polynomials in Theorem~\ref{T:purely geometric bounds2}(b).
We use LMFDB labels for isogeny classes of abelian varieties over finite fields,
formatted as links into the site.

\section{Abelian varieties of order 1}

We say that an abelian variety $A$ over a finite field $\FF_q$ has \emph{order $1$} if we have $\#A(\FF_q) = 1$; that is, the group of $\FF_q$-rational points of $A$ is trivial. Recall that $\#A(\FF_q) = P(1)$ where $P(T) \in \ZZ[T]$ is the Weil polynomial associated to $A$.

\begin{lemma} \label{lem:simple order 1}
Let $A$ be a simple abelian variety of order $1$ over some finite field $\FF_q$.
\begin{enumerate}
    \item[(a)] We must have $q \leq 4$.
    \item[(b)] If $q \in \{3,4\}$, then $A$ is an elliptic curve with Weil polynomial $T^2 - qT + q$.
\item[(c)] If $q=2$, then each root $\alpha$ of the Weil polynomial of $A$ satisfies
    \begin{equation} \label{eq:alpha minpoly}
\alpha^2 + (\eta - 1)\alpha - 2 \eta = 0
\end{equation}
for some root of unity $\eta$. The roots of unity $\eta$ of order $n$  give rise to two irreducible Weil polynomials if $n = 7, 30$
and one otherwise. The resulting $A$ is ordinary unless $n$ is a power of $2$, in which case it has $p$-rank $0$.
\end{enumerate}
\end{lemma}
\begin{proof}
This follows from \cite[Theorem~4]{madan-pal},
\cite{robinson} (for the second assertion of (c)),
\cite[Lemma~5.1]{DWK} (for the description in \eqref{eq:alpha minpoly}), and \cite[Lemma~4.3]{DWK} (for the $p$-rank).
\end{proof}

We deduce some consequences for the Frobenius traces of abelian varieties of order 1; for $q=2$ we establish a stronger result later (Lemma~\ref{lem:excess}).
For $A$ an abelian variety over a finite field $\FF_q$ and $n$ a positive integer, let $T_{A,q^n}$ be the trace of the $q^n$-power Frobenius on $A$; we also write $T_{C,q^n}$ in case $A = J(C)$.

\begin{lemma} \label{lem:simple order 1 traces}
Let $A$ be a simple abelian variety of order $1$ over $\FF_2$. Choose $\alpha, \eta$ as in \eqref{eq:alpha minpoly} and assume that the order of $\eta$ is not in $\{1,2,7,30\}$. 
For
\begin{equation} \label{eq:bound on trace}
t_i = \Trace_{\QQ(\eta)/\QQ}(\eta^i) = \frac{\phi(n)}{\phi(n/\gcd(n,i))} \mu \left( \frac{n}{\gcd(n,i)} \right)
\end{equation}
(where $\mu$ is the M\"obius function), we have
\begin{align*}
T_{A,2} &= \Trace_{\QQ(\eta)/\QQ} (1 - \eta) = \phi(n) - t_1 \\
T_{A,4} &= \Trace_{\QQ(\eta)/\QQ} (1 + 2\eta + \eta^2) = \phi(n) + 2t_1 + t_2 \\
T_{A,8} &= \Trace_{\QQ(\eta)/\QQ} (1 + 3\eta - 3\eta^2 - \eta^3) = \phi(n) + 3t_1 - 3t_2 - t_3 \\
T_{A,16} &= \Trace_{\QQ(\eta)/\QQ} (1 + 4\eta - 2\eta^2 + 4\eta^3 + \eta^4) = \phi(n) + 4t_1 - 2t_2 + 4t_3 + t_4.
\end{align*}
\end{lemma}
\begin{proof}
Our assumption on $n$ ensures that $\QQ(\alpha)$ is a quadratic extension of $\QQ(\eta)$. From \eqref{eq:alpha minpoly}, we see that 
\[
T_{A,2} = \Trace_{\QQ(\alpha)/\QQ}(\alpha) = \Trace_{\QQ(\eta)/\QQ} (1-\eta) = \phi(n) - t_1.
\]
Similarly, from \eqref{eq:alpha minpoly} we deduce that
\begin{align*}
0 &= \alpha^4 + (-1-2\eta - \eta^2)\alpha^2 + 4\eta^2 \\
&= \alpha^6 + (-1 - 3\eta + 3\eta^2 + \eta^3) \alpha^3 - 8\eta^3 \\
&= \alpha^8 + (-1 - 4\eta + 2\eta^2 - 4\eta^3 - \eta^4)\alpha^4 + 16\eta^4,
\end{align*}
from which we read off the expressions for $T_{A,4}, T_{A,8}, T_{A,16}$.
\end{proof}

\begin{lemma} \label{L:t2 plus t4}
Let $A$ be an abelian variety of order $1$ and dimension $g$ over $\FF_q$.
\begin{enumerate}
    \item[(a)] If $q=4$, then $T_{A,q}  = 4g$, $T_{A,q^2} =8g$.
    \item[(b)] If $q=3$, then $T_{A,q} = 3g$, $T_{A,q^2} = 3g$.
    \item[(c)] If $q=2$ and $A$ is simple, then $T_{A,2} + T_{A,4} \geq 2$. This is strict if $g \geq 4$.
\end{enumerate}
\end{lemma}
\begin{proof}
Parts (a) and (b) are apparent from Lemma~\ref{lem:simple order 1}. To check (c), we check for $g \leq 6$ using LMFDB;\footnote{On an LMDFB page, the entry ``Point counts of the curve'' lists $q^i+1-T_{A,q^i}$ for $i=1,\dots,10$.} see Table~\ref{table:simple AVs of order 1} for the detailed results.
For $g > 6$, Lemma~\ref{lem:simple order 1 traces} and \eqref{eq:bound on trace}
yield $T_{A,2} + T_{A,4} = 2g + t_1 + t_2 \geq 2g - 1-2 \geq 2$, as desired.
\end{proof}

\section{Constant extensions}
\label{sec:constant extensions}

In this section, we prove Theorem~\ref{T:constant bounds}. We recall a point from the introduction: for any abelian variety $A$ over $\FF_q$ and any positive integer $d$,
the Weil restriction of $A$ from $\FF_{q^d}$ to $\FF_q$ is isogenous to the product of $A$ with the ``Prym variety'' $A'$.

\begin{lemma} \label{L:base change degree 3}
Let $A$ be an abelian variety over $\FF_q$ such that $\#A(\FF_q) = \#A(\FF_{q^d})$ for some prime $d>2$. Then $q=2$, $d=3$, and
the Weil polynomial of every simple isogeny factor of $A$ belongs to $\{T^2 + T + 2, T^2 + 2T + 2\}$.
\end{lemma}
\begin{proof}
Since $[A(\FF_{q^d}):A(\FF_q)] = \#A'(\FF_q)$ is an integer, the hypothesis that $\#A(\FF_q) = \#A(\FF_{q^d})$ implies the same for the isogeny factors of $A$;
we may thus assume that $A$ is simple.
Let $P(T)$ be the Weil polynomial of $A$. Then 
the Weil polynomial of $A'$ is $\prod_{i=1}^{d-1} P(\zeta_d^i T)$, and hence has roots $\alpha_1,\dots,\alpha_{d-1}$ such that
\[
\alpha_1 \zeta_d = \cdots = \alpha_{d-1} \zeta_d^{d-1};
\]
by Lemma~\ref{lem:simple order 1}, this is impossible if $q > 2$.
If $q=2$, then by \eqref{eq:alpha minpoly} there must exist roots of unity $\eta_1,\dots,\eta_{d-1}$ with
\[
\alpha_i^2 + (\eta_i-1)\alpha_i - 2\eta_i = 0 \qquad (i=1,\dots,d-1).
\]
For $1\leq i < j \leq d-1$, applying \cite[Lemma~5.2, Lemma 7.2]{DWK} to the equation $\alpha_i = \alpha_j \zeta_d^{j-i}$
shows that $(\eta_i, \eta_j, \zeta_d^{j-i})$ either appears in one of the parametric solutions in \cite[(7.2.1)]{DWK}
or is a sporadic solution fitting a pattern listed in \cite[Table~2]{DWK}.

If only parametric solutions occur, then from \cite[(7.2.1)]{DWK} we have 
$\eta_1 = \cdots = \eta_{d-1}$, leaving only two distinct values for $\alpha_1,\dots,\alpha_{d-1}$.
Hence $d= 3$; from \cite[(7.2.1)]{DWK} again, $\eta_1 = \eta_2 = -\zeta_3$ has order 6. This yields the Weil polynomial
$T^2+T+2$.

If we get a sporadic solution for some $i,j$, then \cite[Table~2]{DWK} indicates that $\zeta_d^{j-i}$ has order dividing 21, 24, or 30; this forces $d \leq 7$.
For $d \in \{5,7\}$, the $\eta_i$ must all have order 30 or 7, respectively; however, if $\alpha$ satisfies \eqref{eq:alpha minpoly} for some root of unity $\eta$ of this order, then at most two of the quantities $\{\alpha \zeta_d^i: i=1,\dots,d-1\}$ do likewise, and this leaves no options for $A'$.
Hence $d=3$; from \cite[Table~2]{DWK} (taking $\eta_3 = \zeta_3$), $\eta_1=\eta_2$ has order 4.
This yields the Weil polynomial $T^2+2T+2$.
\end{proof}

\begin{lemma} \label{L:constant classification}
Let $C$ be an algebraic curve of genus $g>0$ over $\FF_q$ such that $\#J(C)(\FF_q) = \#J(C)(\FF_{q^d})$ for some integer $d>1$.
Then
\[
(q,d,g) \in \{(2,2,1), (2,2,2), (2,2,3), (2,3,1), (3,2,1), (4,2,1)\}.
\]
\end{lemma}
\begin{proof}
It suffices to prove the claim when $d$ is prime, as the result will then rule out composite values of $d$.
By Lemma~\ref{lem:simple order 1}, $q \leq 4$. 
By Lemma~\ref{L:base change degree 3} applied with $A = J(C)$,
if $d > 2$ then $(q,d,g) = (2,3,1)$.

Assume now that $d=2$.
Then the Prym variety $A'$ is the quadratic twist of $J(C)$, so $T_{A',q^i} = (-1)^i T_{C,q^i}$.
If $q \in \{3,4\}$, then by Lemma~\ref{L:t2 plus t4},
\begin{align*}
0 &\leq\#C(\FF_{q^2}) - \#C(\FF_q) = (q^2 + 1 - T_{A',q}) - (q + 1 + T_{A',q}) \\
& = q^2-q - T_{A',3} - T_{A',9} = q^2-q - q(q-1)g
\end{align*}
and so $g \leq 1$.
If $q=2$, then
\[
0 \leq \#C(\FF_4) - \#C(\FF_2) = (2^2 + 1 - T_4) - (2 + 1 + T_2) = 2 - T_2 - T_4 \leq 0
\]
with the last inequality strict unless $A'$ is simple of dimension at most 3.
\end{proof}

\begin{lemma} \label{L:constant classification full}
Let $C$ be a curve over $\FF_q$ such that $\#J(C)(\FF_q) = \#J(C)(\FF_{q^d})$ for some $d>1$.
Then $C$ appears in Theorem~\ref{T:constant bounds}.
\end{lemma}
\begin{proof}
As this property only depends on the isogeny class of $J(C)$, it suffices to search over the isogeny classes in LMFDB permitted by Lemma~\ref{L:constant classification}.
\end{proof}

\section{Bounds on rational points on curves}

We next compile some explicit upper bounds for the number of rational points on a curve over $\FF_q$. 
For $g \leq 10$, we reproduce in Table~\ref{table:manypoints bounds} some data from \cite{manypoints}
(see therein for underlying references).
For larger $g$, we use the ``linear programming'' method of Oesterl\'e.
(All decimal expansions herein refer to \emph{exact} rational numbers.)

\begin{table}[ht]
\tiny
\begin{tabular}{c||c|c|c|c|c|c|c|c|c|c}
$g$ & 1 & 2 & 3 & 4 & 5 & 6 & 7 & 8 & 9 & 10 \\
\hline
$q=2$ & 5 & 6 & 7 & 8 & 9 & 10 & 10 & 11 & 12 & 13 \\
$q=2^2$  & 9 & 10 & 14 & 15 & 17 & 20 & 21 & 23 & 26 & 27 \\
$q=2^3$ & 25 & 33 & 38 & 45 & 53 & 65 & 69 & 75 & 81 & 86 \\
\hline
$q=3$ & 7 & 8 & 10 & 12 & 13 & 14 & 16 & 18 & 19 & 21 \\
$q=3^2$ & 16 & 20 & 28 & 30 & 35 & 38 & 43 & 46 & 50 & 54
\end{tabular}
\medskip
\caption{Upper bounds on $\#C(\FF_q)$ for a genus-$g$ curve $C$ from \cite{manypoints}.}
\label{table:manypoints bounds}
\end{table}

\begin{lemma} \label{lem:oesterle}
Let $C$ be a curve of genus $g$ over $\FF_q$ with $q \in \{2,3,4\}$. Then
\[
\#C(\FF_q) \leq \begin{cases} 0.6272g + 9.562 & (q=2) \\
1.153g + 11.67 & (q=3) \\
1.435g + 21.75 & (q=4).
\end{cases}
\]
\end{lemma}
\begin{proof}
For $q=2$, this is the ``third choice'' bound of \cite[(7.1.4)]{serre-rational}. For $q=3,4$,
we adapt the proof of the ``first choice'' bound of \cite[(7.1.1)]{serre-rational}).
For $x_1, x_2, \ldots \geq 0$, define $c = 1 + 2x_1^2 + 2x_2^2 + \cdots$ and consider the function
\[
f(\theta) = \frac{1}{c} (1 + 2x_1 \cos(\theta) + 2x_2 \cos(2\theta) + \cdots)^2 = 1 + 2 \sum_{n \geq 1} c_n \cos(n \theta).
\]
By construction, $f(\theta) \geq 0$ for all $\theta \in \RR$ and $c_n \geq 0$ for all $n$ (that is, $f$ is \emph{doubly positive} in the sense of Serre).
Define $\psi(t) = \sum_{n=1}^\infty c_n t^n$; then by \cite[Theorem~5.3.3]{serre-rational}. 
\[
\#C(\FF_q)\psi(q^{-1/2}) \leq g + \psi(q^{-1/2}) + \psi(q^{1/2}),
\]
or in other words
\[
\#C(\FF_q) \leq \frac{1}{\psi(q^{-1/2})} g + 1 + \frac{\psi(q^{1/2})}{\psi(q^{-1/2})}.
\]
For $x_1 = 1, x_2 = 0.7, x_3 = 0.2, x_4 = \cdots = 0$,
this yields the indicated results.
\end{proof}

The bounds produced by linear programming also include some correction terms counting points over extension fields.
We make one such bound explicit for $q=2$.
\begin{lemma} \label{lem:oesterle2}
Let $C$ be a curve of genus $g$ over $\FF_2$. For $d=1,2,\dots$, let $a_d$ be the number of closed points of degree $d$ on $C$. Then
\begin{equation} \label{eq:refined sum}
a_1 + 2a_2 (0.3366)+ 3a_3 (0.1382) + 4a_4 (0.0537)
\leq 0.8042g + 5.619.
\end{equation}
\end{lemma}
\begin{proof}
With notation as in the proof of Lemma~\ref{lem:oesterle}, define $\psi_d(t) = \sum_{n=1}^\infty c_{dn} t^{dn}$.
Then by \cite[Theorem~5.3.3]{serre-rational} again,
\[
\sum_{d=1}^\infty d a_d \psi_d(q^{-1/2}) \leq g + \psi(q^{-1/2}) + \psi(q^{1/2}),
\]
or in other words
\begin{equation} \label{eq:refined sum parametric}
a_1 + \sum_{d=2}^\infty d a_d \frac{\psi_d(q^{-1/2})}{\psi(q^{-1/2})} \leq \frac{1}{\psi(q^{-1/2})} g + 1 + \frac{\psi(q^{1/2})}{\psi(q^{-1/2})}.
\end{equation}
We apply this with $x_1 = 1, x_2 = 0.85, x_3 = 0.25, x_4 = \cdots = 0$.
This yields \eqref{eq:refined sum} by discarding the terms $d \geq 5$ in \eqref{eq:refined sum parametric}.
\end{proof}

\section{Numerical estimates}

We next apply the bounds on rational points to bound the genera of function fields occurring in a purely geometric extension with relative class number 1. We will later take a closer account of the degree of the extension; see
\S\ref{sec:analysis by degree}.

For the remainder of the paper, let $F'/F$ be a purely geometric extension of degree $d$
such that $g_{F'} > g_F$ and $h_{F'/F} = 1$. 
For brevity, we write $q,g,g'$ in place of $q_F, g_F, g_{F'}$.
Let $A$ be the Prym variety of $C' \to C$; 
then $A$ has order 1, so Lemma~\ref{lem:simple order 1} implies $q \leq 4$. By Riemann--Hurwitz,
\begin{equation} \label{eq:riemann-hurwitz1}
\dim(A) = g'-g = (d-1)(g-1) + \delta \qquad \mbox{with} \qquad \delta \geq 0,
\end{equation}
with equality if and only if $C' \to C$ is \'etale.
Since $T_{C',q^i} = T_{C,q^i} + T_{A,q^i}$, we have
\begin{equation} \label{eq:C to Cprime count}
0 \leq \#C'(\FF_{q^i}) = q^i + 1 - T_{C',q^i} = q^i+1-T_{C,q^i} - T_{A,i} = \#C(\FF_{q^i}) - T_{A,q^i}
\end{equation}
for each positive integer $i$, and hence
\begin{equation} \label{eq:separate curve count}
T_{A,q^i} \leq \#C(\FF_{q^i}) \qquad (i=1,2,\dots).
\end{equation}

\begin{lemma} \label{lem:q34 genus bound}
If $q > 2$, then $g \leq 6$.
\end{lemma}
\begin{proof}
By combining Lemma~\ref{L:t2 plus t4}, Lemma~\ref{lem:oesterle}, \eqref{eq:riemann-hurwitz1}, and \eqref{eq:separate curve count}, we obtain
\begin{equation} \label{eq:main bound}
q(g-1) \leq q(g'-g) \leq \#C(\FF_q) \leq  \begin{cases} 1.153g + 11.67 & (q=3) \\
1.435g + 21.75 & (q=4). \end{cases}
\end{equation}
Comparing the ends of this equation yields
\[
g \leq \begin{cases} (11.67 + 3)/(3 - 1.153) \leq 7.95 & (q=3) \\ (21.75 + 4)/(4 - 1.435) \leq 10.04 & (q=4);
\end{cases}
\]
hence $g \leq 7$ if $q=3$ and $g \leq 10$ if $q = 4$. Replacing the right-hand side of \eqref{eq:main bound} with the explicit bounds given in Table~\ref{table:manypoints bounds}, 
we may eliminate the case $g = 7$. 
\end{proof}

For $q=2$, it is not enough to control $\#C(\FF_2)$ because there exists a simple abelian variety of order 1 with trace 0 (namely \avlink{2.2.a\_ae}). Instead, we use a bound modeled on Lemma~\ref{lem:oesterle2}.
For $A$ an abelian variety over $\FF_2$, define its \emph{excess} as
\[
1.3366 T_{A,2} + 0.3366 T_{A,4} + 0.1137 (T_{A,8} - T_{A,2}) + 0.0537 (T_{A,16} - T_{A,4}) - 1.5612g.
\]

\begin{lemma} \label{lem:excess}
For $A$ an abelian variety of order $1$ and dimension $g$ over $\FF_2$, the excess of $A$ is nonnegative.
\end{lemma}
\begin{proof}
We may assume that $A$ is simple; define $n$ as in Lemma~\ref{lem:simple order 1}.
We again treat the case $g \leq 6$ using LMFDB; see Table~\ref{table:simple AVs of order 1}.
For $g \geq 7$, we have $g = \phi(n)$; per Lemma~\ref{lem:simple order 1 traces} we can write the excess as
\[
0.112g - 0.1012t_1 - 0.1656t_2 + 0.1011t_3 + 0.0537t_4.
\]
For $g \in \{7,8\}$, we have $n \in \{15,16,20,24,30\}$;
we compute the excess using \eqref{eq:bound on trace} to obtain a lower bound of $0.4807$.
For $g \geq 9$, we apply \eqref{eq:bound on trace} to deduce that $|t_d| \leq d$
and then obtain a lower bound of $0.112g - 0.9505 \geq 0.112\cdot 9 - 0.9505 \geq 0.0575$. 
\end{proof}

\begin{table}[ht]
\tiny
\begin{tabular}{c||c||c|c|c|c||c|c}
$A$ & $n$ & $T_{A,2}$ & $T_{A,4}$ & $T_{A,8}$ & $T_{A,16}$ & $T_{A,2} + T_{A,4}$ & excess\\
\hline
\hline
\avlink{1.2.ac} & $2$ & $2$ & $0$ & $-4$ & $-8$ & $2$ & 0.0002 \\
\hline
\avlink{2.2.a\_ae} & $1$ & $0$ & $8$ & $0$ & $16$ & $8$ & 0.0000 \\
\avlink{2.2.ad\_f} & $3$ & $3$ & $-1$ & $0$ & $7$ & $2$ & 0.6393 \\
\avlink{2.2.ac\_c} & $4$ & $2$ & $0$ & $8$ & $8$ & $2$ & 0.6626 \\
\avlink{2.2.ab\_ab} & $6$ & $1$ & $3$ & $10$ & $-1$ & $4$ & 0.0325 \\
\hline
\avlink{3.2.ad\_c\_b} & $7$ & $3$ & $5$ & $6$ & $-11$ & $8$ & 0.4911 \\
\avlink{3.2.ae\_j\_ap} & $7$ & $4$ & $-2$ & $1$ & $10$ & $2$ & 0.2929 \\
\hline
\avlink{4.2.af\_m\_au\_bd} & $5$ & $5$ & $1$ & $5$ & $-3$ & $6$ & 0.5600 \\
\avlink{4.2.ae\_g\_ae\_c} & $8$ & $4$ & $4$ & $4$ & $0$ & $8$ & 0.2332 \\
\avlink{4.2.ad\_c\_a\_b} & $10$ & $3$ & $5$ & $9$ & $13$ & $8$ & 0.5598 \\
\avlink{4.2.ae\_f\_c\_al} & $12$ & $4$ & $6$ & $-2$ & $-2$ & $10$ & 0.0094 \\
\avlink{4.2.ae\_e\_h\_av} & $30$ & $4$ & $8$ & $-5$ & $4$ & $12$ & 0.5563 \\
\avlink{4.2.af\_n\_az\_bn} & $30$ & $5$ & $-1$ & $5$ & $7$ & $4$ & 0.5312 \\
\hline
\avlink{6.2.ag\_p\_av\_y\_abn\_cn} & $9$ & $6$ & $6$ & $9$ & $-6$ & $12$ & 0.3687 \\
\avlink{6.2.af\_j\_ah\_d\_ab\_ab} & $14$ & $5$ & $7$ & $11$ & $15$ & $12$ & 0.7838 \\
\avlink{6.2.ag\_p\_at\_g\_bb\_acj} & $18$ & $6$ & $6$ & $3$ & $18$ & $12$ & 0.9753 \\
\end{tabular}
\medskip
\caption{Simple abelian varieties over $\FF_2$ of order 1 and dimension at most 6, from \href{https://www.lmfdb.org/Variety/Abelian/Fq/?q=2&simple=yes&g=1-6&abvar_point_count=\%5B1\%5D&search_type=List}{LMFDB}. For the definitions of $n$ and the excess, see Lemma~\ref{lem:simple order 1} and  Lemma~\ref{lem:excess}.}
\label{table:simple AVs of order 1}
\end{table}

\vspace{-0.5cm}

\begin{lemma} \label{lem:genus bound q2}
For $q=2$, we have
\[
g' \leq 0.4313 \#C(\FF_2) + 1.5152g + 3.6.
\]
\end{lemma}
\begin{proof}
We combine Lemma~\ref{lem:oesterle2}, \eqref{eq:separate curve count}, and Lemma~\ref{lem:excess} to obtain
\begin{align*}
1.5612 (g'-g) &\leq 1.3366 T_{A,q} + 0.3366 T_{A,q^2} + 0.1137 (T_{A,q^3} - T_{A,q}) \\
&\qquad + 0.0537 (T_{A,q^4} - T_{A,q^2}) \\ 
&= (1.3366 - 0.1137) T_{A,q} + (0.3366 - 0.0537) T_{A,q^2} \\
&\qquad + 0.1137 T_{A,q^3} + 0.0537 T_{A,q^4} \\ 
&\leq (1.3366 - 0.1137) \#C(\FF_q) + (0.3366 - 0.0537) \#C(\FF_{q^2}) \\
&\qquad + 0.1137 \#C(\FF_{q^3}) + 0.0537 \#C(\FF_{q^4}) \\ 
&= 1.3366 a_1 + 0.3366 (a_1 + 2a_2) + 0.1137 (3a_3) + 0.0537 (4a_4) \\
&= 1.6732a_1 + 0.3366 (2a_2) + 0.1137 (3a_3) + 0.0537 (4a_4) \\
&\leq 0.6732\#C(\FF_2) + 0.8042g + 5.619,
\end{align*}
which yields the claimed inequality.
\end{proof}

\begin{cor} \label{cor:genus bound q2}
For $q=2$, we have $g \leq 40$. Moreover, if $d \geq 3$ then $g \leq 6$; if $d \geq 4$ then $g \leq 4$; if $d \geq 5$ then $g \leq 3$; and if $d \geq 6$ then $g \leq 2$.
\end{cor}
\begin{proof}
By \eqref{eq:riemann-hurwitz1} and Lemma~\ref{lem:genus bound q2},
\[
(d - 1.5152)g \leq 0.4313 \#C(\FF_2) + (d + 2.6).
\]
Taking $d = 2$ and using the bound on $\#C(\FF_2)$ from Lemma~\ref{lem:oesterle} yields $g \leq 40$.
For $d \geq 3$ we obtain $g \leq 8$; we then use Table~\ref{table:manypoints bounds} to obtain the remaining bounds.
\end{proof}

\section{Exhaustion over Weil polynomials}
\label{sec:exhaustion}

We next describe an exhaustive search over Weil polynomials which rules out some additional pairs $(g,g')$; compare \cite[Theorem~7.2.1]{serre-rational} for an example in the context of bounding rational points on curves. This will yield Theorem~\ref{T:purely geometric bounds1};
for $g > 1$, we will do better with constraints depending on $d$ (see \S\ref{sec:analysis by degree}).

We first make a list of candidate Weil polynomials for $A$. For $q > 2$ this consists of the single polynomial $(T^2 - qT + q)^{g'-g}$.
For $q=2$, we identify isogeny classes of simple abelian varieties $A$ of order 1 such that
for $i=1,2$, $T_{A,2^i}$ and $T_{A,4}$ is at most the value listed in Table~\ref{table:manypoints bounds} for the pair $(g, 2^i)$, and moreover
\begin{align*}
&{T_{A,2} + 0.3366 (T_{A,4}-T_{A,2}) + 0.1137 (T_{A,8} - T_{A,2}) + 0.0537 (T_{A,16} - T_{A,4})} \\
&\qquad \leq 0.8042g + 5.619;
\end{align*}
these are all necessary conditions by Lemma~\ref{lem:oesterle2} and \eqref{eq:separate curve count}.

We next identify candidate Weil polynomials for $C$ for which the resulting values of $\#C(\FF_q)$ (and $\#C(\FF_{q^2})$ for $q=2$) are consistent with at least one choice of $A$,
and eliminate those that are ruled out by any of the following.
\begin{itemize}
\item
Bounds on point counts from Table~\ref{table:manypoints bounds}.
\item
The \emph{positivity condition}: the number of degree-$i$ places on $C$ must be nonnegative for all $i \geq 1$.
\item
Data from LMFDB (genus $\leq 3$), \cite{xarles} (genus 4), and \cite{dragutinovic} (genus 5)
indicating which curves have a particular Weil polynomial.
\item
The \emph{resultant-$1$ and resultant-$2$ criteria} of Serre \cite[Theorem~2.4.1]{serre-rational} as extended by Howe--Lauter \cite[Proposition~2.8]{howe-lauter2}. The resultant-2 criterion forces $C$ to occur as a double cover of another curve, whose Weil polynomial can sometimes be ruled out.
(Compare Corollary~\ref{cor:Delta sequence}.)
\end{itemize}
Finally, we exhaust over pairs of candidate Weil polynomials for $C$ and $A$ to confirm that the resulting Weil polynomial for $C'$ is not ruled out. This yields the following.

\begin{lemma} \label{lem:genus bound q2 refinement}
For $q= 2$, for $g = 0,\dots,6$ we have $g' \leq 4,6,8,10,12,14,16$, respectively.
Hence by \eqref{eq:riemann-hurwitz1}, if $d \geq 4$ then $g \leq 3$; and if $d \geq 5$ then $g \leq 2$.
\end{lemma}
\begin{proof}
From Lemma~\ref{lem:genus bound q2}, for $g=0,\dots,6$ we obtain $g' \leq 4, 7, 9, 11, 13, 15, 17$, respectively.
We rule out the pairs 
$(g,g') \in \{(1, 7), (2, 9), (3, 11), (4, 13), (5, 15), (6, 17) \}$
by exhausting over Weil polynomials as described above.
\end{proof}

We can now prove Theorem~\ref{T:purely geometric bounds1} as follows.
By \eqref{eq:main bound}, Lemma~\ref{lem:genus bound q2 refinement}, and Table~\ref{table:manypoints bounds},
for $(q,g) = (2,0), (2,1), (3,0), (3,1), (4,0), (4,1)$ we have respectively
\[
g' \leq 4, 6, 1, 3, 1, 3.
\]
We may settle all cases by table lookups except $(q,g,g') = (2,1,6)$, which we settle as follows.
\begin{itemize}
\item
The isogeny class \avlink{6.2.ad\_c\_a\_a\_m\_abg} is ruled out by \cite[Proposition~5.2]{faber-grantham-howe}, whose proof we summarize.
By the resultant-$2$ criterion (compare Remark~\ref{R:resultant degree 2}), $C'$ is a double cover of a curve $C_0$ with real Weil polynomial $T^2 - 2T - 2$; this is inconsistent with $\#C_0(\FF_2) = 1, \#C'(\FF_4) = 0$.
\item
The isogeny class \avlink{6.2.ad\_c\_a\_f\_am\_q} occurs for a cyclic \'etale quintic cover of a genus-2 curve listed in Table~\ref{table:geometric extensions big q2b1}
(see also Remark~\ref{R:unique curve}).
\end{itemize}

\begin{remark} \label{R:unique curve}
Table~\ref{table:geometric extensions big q1} includes a column counting Jacobians in the isogeny class of $J(C')$.
This can be obtained by table lookups except for \avlink{6.2.ad\_c\_a\_f\_am\_q}, for which Table~\ref{table:geometric extensions big q1} reports a \emph{unique} Jacobian; this will be proved in \cite[Lemma~10.2]{part2}.
\end{remark}

\section{Additional constraints on Weil polynomials}
\label{sec:analysis by degree}

We assume hereafter that $g > 1$
and introduce constraints on the Weil polynomials of $C$ and $C'$ based on $d$.
Note that none of these presumes $h_{F'/F} = 1$, and so may be applicable in other cases of interest.

We start with the full form of Riemann--Hurwitz:
\begin{equation} \label{eq:riemann-hurwitz}
2g'-2 = d(2g-2) + 2\delta, \qquad 2\delta = \sum_P (e_P - 1)
\end{equation}
where $P$ runs over geometric points of $C'$ and $e_P$ is the ramification index at $P$.

Let $t$ denote the number of geometric ramification points, i.e., the number of $P$ for which $e_P > 1$.
Then $t=0$ iff $\delta = 0$,
and $t \leq 2\delta$ in general. If $q$ is even, then $e_P$ can never equal 2, so
$t \leq \delta$; in particular,
\begin{equation} \label{eq:unique ramification point}
\delta = 1 \Longrightarrow t = 1 \Longrightarrow \#C'(\FF_q) \geq 1
\end{equation}
because the unique ramification point of $C'$ is $\FF_q$-rational, and similarly
\begin{equation} \label{eq:unique ramification point2}
\delta=2\Longrightarrow \#C'(\FF_{q^2}) \geq 1,
\qquad
t=2 \Longrightarrow \#C'(\FF_{q^2}) \geq 2.
\end{equation}

If $C' \to C$ is cyclic of prime degree $d = p \mid q$, then
the Deuring--Shafarevich formula holds (e.g., see \cite{shiomi}): for $\gamma_C, \gamma_{C'}$ the $p$-ranks of $C, C'$,
\begin{equation} \label{eq:deuring-shafarevich}
\gamma_{C'}-1 = d(\gamma_C - 1) + t
\end{equation}

If $\delta=0$ and $C' \to C$ is cyclic (e.g., if $d=2$), then by class field theory,
\begin{equation} \label{eq:cyclic}
\#J(C)(\FF_q) \equiv 0 \pmod{d}.
\end{equation}

For small $d$, we have the following additional constraints (building on \cite[Lemma~8]{howe-lauter1}).
\begin{itemize}
\item
When $d=2$, every $\FF_{q^i}$-rational point of $C$ lifts to either an $\FF_{q^i}$-rational ramification point or two $\FF_{q^{2i}}$-rational points of $C'$. Hence 
\begin{equation} \label{eq:double cover without trace}
\#C'(\FF_{q^{2i}}) \geq 2\#C(\FF_{q^i}) - t;
\end{equation}
by \eqref{eq:C to Cprime count} and \eqref{eq:separate curve count}, this yields
\begin{equation} \label{eq:double cover consequence}
2T_{A,q^i} + T_{A,q^{2i}} - t \leq 2\#C(\FF_{q^i}) + T_{A,q^{2i}} - t \leq \#C(\FF_{q^{2i}}).
\end{equation}
For $i = 2j-1$ odd, every degree $i$-place of $C'$ projects to a degree-$i$ place of $C$.
If $t \leq 2$, then for $i>1$ these points occur in pairs in fibers, and so
\begin{equation} \label{eq:double cover parity}
t \leq 2 \Longrightarrow \#C'(\FF_{q^{2j-1}}) \equiv \#C'(\FF_q) \pmod{2} \qquad (j > 0).
\end{equation}
\item
When $d=3$, every $\FF_{q^i}$-rational point of $C$ lifts to either at least one $\FF_{q^i}$-rational point or three $\FF_{q^{3i}}$-rational points of $C'$.
Hence $\#C'(\FF_{q^{3i}}) - \#C'(\FF_{q^i}) \geq 3(\#C(\FF_{q^i}) - \#C'(\FF_{q^i}))$;
by \eqref{eq:separate curve count}, this yields
\begin{equation} \label{eq:triple cover consequence}
\#C(\FF_{q^i}) + 2T_{A,q^i} + T_{A,q^{3i}} \leq \#C(\FF_{q^{3i}}).
\end{equation}

\item
When $d=4$, every $\FF_{q^i}$-rational point of $C$ lifts to at least one $\FF_{q^i}$-rational point, two $\FF_{q^{2i}}$-rational ramification points, or four
$\FF_{q^{4i}}$-rational points of $C'$. Hence $\#C'(\FF_{q^{4i}})  \geq 4(\#C(\FF_{q^i}) - \#C'(\FF_{q^i})) - 2t$;
by \eqref{eq:separate curve count}, this yields
\begin{equation} \label{eq:quadruple cover consequence}
4T_{A,q^i} + T_{A,q^{4i}} -2\delta \leq 4T_{A,q^i} + T_{A,q^{4i}} -2t \leq \#C(\FF_{q^{4i}}).
\end{equation}

\end{itemize}

\begin{remark} \label{R:relative quadratic twist}
For $d=2$, the compositum $F' \cdot \FF_{q^2}$ contains another purely geometric quadratic extension $F''/F$. We call the corresponding cover $C'' \to C$ the \emph{relative quadratic twist}
of $C' \to C$; it also obeys the conditions listed in \S\ref{sec:exhaustion}.
\end{remark}

\section{Purely geometric extensions: \texorpdfstring{$q > 2$}{q>2}}
\label{sec:geometric q34}

We settle Theorem~\ref{T:purely geometric bounds2}(a) as follows. 
For $q > 2$, Lemma~\ref{lem:q34 genus bound} implies $g \leq 6$.
If $d=2$, then by Lemma~\ref{L:t2 plus t4} plus \eqref{eq:double cover consequence},
\begin{equation} \label{eq:double cover2}
\#C(\FF_{q^2}) \geq 2T_{A,q} + T_{A,q^2} - c(g'-2g+1) = q^2(g'-g) - c(g'-2g+1) \geq q^2(g-1).
\end{equation}
Combining \eqref{eq:double cover2} with Table~\ref{table:manypoints bounds}, we deduce that
\begin{gather*}
(q,g,g') \in \{(3,2,3), (3,2,4), (3,3,5), (3,3,6), (3,4,7), (4,2,3), (4,2,4), (4,3,5)\}.
\end{gather*}
If $d > 2$, then by upgrading \eqref{eq:main bound} using Table~\ref{table:manypoints bounds},
we deduce that $(d,g,g') = (3,2,4)$. We also have the following.
\begin{lemma} \label{L:degree 3 cyclic}
If $q>2$, $g=2$, and $d=3$, then $C' \to C$ is cyclic.
\end{lemma}
\begin{proof}
Suppose first that $C' \to C$ is a non-Galois cover which becomes Galois after a quadratic constant field extension.
By Lemma~\ref{L:t2 plus t4}, the quadratic twist $\tilde{C}$ of $C$ admits a cyclic cubic \'etale cover $\tilde{C}'$ whose Prym has Weil polynomial $(T^2 + qT + q)^2$. 
Since each $\FF_q$-point of $\tilde{C}$ lifts to at most three $\FF_q$-points of $\tilde{C}'$, we have $\#\tilde{C}(\FF_q) + 2q = \tilde{C}'(\FF_q) \leq 3 \#\tilde{C}(\FF_q)$
and so $\tilde{C}(\FF_q) \geq q$. However, $\#C(\FF_q) \geq 2q$ by \eqref{eq:separate curve count}, yielding the impossibility
\[
2q+2 = \#C(\FF_q) + \#\tilde{C}(\FF_q) \geq 3q.
\]

Suppose next that $C' \to C$ is geometrically non-Galois.
In this case, the Galois closure $F''$ of $F'/F$ is itself the function field of a curve $C''$
with $q_{F''} = q_F$. The abelian variety
$J(C'')$ is isogenous to $J(C) \times A^2 \times E$ for some elliptic curve $E$, so
\[
\#C''(\FF_q) = \#C(\FF_q) - 2T_{A,q} - T_{E,q} = \#C(\FF_q) - 4q - T_{E,q} \leq \#C(\FF_q) - 3q;
\]
this yields $\#C(\FF_q) \geq 3q$, which is inconsistent with  Table~\ref{table:manypoints bounds}.
\end{proof}

We now know that in all cases $C' \to C$ is cyclic, so we may proceed as follows.
\begin{itemize}
\item
We again exhaust over Weil polynomials for $C$ and $A$, but this time accounting for
\eqref{eq:deuring-shafarevich}, \eqref{eq:cyclic}, \eqref{eq:double cover consequence},
\eqref{eq:triple cover consequence}, \eqref{eq:double cover2}. 
At this point the cases $(q,d,g,g') = (3,2,3,6), (4,2,2,4)$ drop out.
\item
For each candidate Weil polynomial for $C$, we consult LMFDB to find all candidates for $C$.
At this point the case $(q,d,g,g') = (3,2,4,7)$ drops out: the only isogeny class for $J(C)$ is \avlink{4.3.f\_v\_ca\_eg}, which contains no Jacobian.
\item
We then use \Magma{} to compute all cyclic extensions of $F$ with the desired degree and ramification behavior
and check the resulting Weil polynomial for $A$. At this point the case $(q,d,g,g') = (4,2,3,5)$ drops out.
\end{itemize}
This yields Theorem~\ref{T:purely geometric bounds2}(a).
 
\section{A refined resultant criterion}

In preparation for the case $q=2$, we next introduce a refinement of the resultant criteria, modeled on \cite[Proposition~2.8]{howe-lauter2} (applicable over any finite base field).
\begin{lemma} \label{lem:Delta sequence}
Let $f: C' \to C$ be a finite flat morphism of degree $d$ between smooth projective curves over an arbitrary field $k$.
Let $f^*: J(C) \to J(C')$ denote the pullback map and
let $f_*: J(C') \to J(C)$ denote the pushforward map.
Let $A$ be the Prym variety of $f$, defined as the reduced closed subscheme of the identity component of $\ker(f_*)$. Then there is an exact sequence
\begin{equation} \label{eq:Delta sequence}
0 \to \Delta \to J(C) \times_k A \to J(C') \to 0
\end{equation}
where the map $J(C) \to J(C')$ is $f^*$ and $\Delta$ is a finite flat group scheme killed by $d$.
\end{lemma}
\begin{proof}
The composition $J(C) \stackrel{f^*}{\to} J(C') \stackrel{f_*}{\to} J(C)$ equals the isogeny $[d]$; consequently, $f_*$ is surjective (as a morphism of group schemes) and $J(C) \stackrel{f^*}{\to} J(C') \to J(C')/\ker(f_*)$ is surjective.
The latter implies that $J(C) \times_k \ker(f_*) \to J(C')$ is surjective,
as then is $\ker(f_*) \to \coker(f^*)$; since the target is connected and reduced, $A \to \coker(f^*)$ is surjective,
as then is $J(C) \times_k A \to J(C')$.

Let $S$ be an arbitrary $k$-scheme and suppose $x \in (J(C) \times_k A)(S)$ maps to zero to $J(C')$.
Write $x = (x_1, x_2)$ with $x_1 \in J(C)(S)$ and $x_2 \in A(S)$. By definition, $x_1$ and $-x_2$ have the same image in
$J(C')(S)$; that is, $f^*(x_1) = -x_2$. Applying $f_*$, we deduce that $f_* f^* (x_1) = 0$, and so $[d](x_1) = 0$;
it follows that $[d](x) = ([d](x_1), [d](x_2)) = (0, [d](x_2))$ maps to  zero in $J(C')$, and hence $[d](x_2) = 0$.
\end{proof}

\begin{cor} \label{cor:Delta sequence}
In Lemma~\ref{lem:Delta sequence},
let $h_1$ and $h_2$ be the radicals of the real Weil polynomials associated to $J(C)$ and $A$. 
Let $\widetilde{\res}(h_1,h_2)$ be the modified reduced resultant of $h_1$ and $h_2$ in the sense of \cite[Proposition~2.8]{howe-lauter2}. Then 
\begin{equation} \label{eq:cor Delta sequence}
\gcd(d, \widetilde{\res}(h_1,h_2)) > 1.
\end{equation}
\end{cor}
\begin{proof}
In \eqref{eq:Delta sequence}, $\Delta$ cannot be trivial: otherwise, $J(C')$
would be decomposable as a principally polarized abelian variety, violating Torelli \cite[Theorem~12.1]{milne}. The exponent of $\Delta$ divides $d$
by Lemma~\ref{lem:Delta sequence} and $\widetilde{\res}(h_1,h_2)$ by \cite[Proposition~2.8]{howe-lauter2}.
\end{proof}

\section{Purely geometric extensions: \texorpdfstring{$q=2$}{q=2}}
\label{sec:purely geometric q2}

To conclude, we establish parts (b) and (c) of Theorem~\ref{T:purely geometric bounds2}.

\begin{lemma} \label{lem:degree 2 bound1}
If $q=2$ and $d=2$, then $g \leq 9$. Moreover, for $g=2,\dots,9$ we have respectively $g' \leq 7, 9, 10, 11, 13, 14, 15, 17$.
\end{lemma}
\begin{proof}
Combining  \eqref{eq:riemann-hurwitz1}, Lemma~\ref{lem:excess},and \eqref{eq:double cover consequence} yields
\begin{align*}
1.5612(g-1) &\leq 1.5612(g'-g) - 0.3366 \delta \\
&\leq T_{A,q} + 0.3366 (T_{A,q} + T_{A,q^2} - \delta) \\
&\qquad + 0.1137 (T_{A,q^3} - T_{A,q}) + 0.0537 (T_{A,q^4} - T_{A,q^2}) \\
&\leq (1 - 0.3366 - 0.1137 + 0.0537) T_{A,q}  \\
&\qquad + (0.3366-0.0537) (T_{A,q} + T_{A,q^2} - t) + 0.1137 T_{A,q^3} + 0.0537 T_{A,q^4} \\
&\leq (1 - 0.3366 - 0.1137 + 0.0537) \#C(\FF_q)  \\
&\qquad + (0.3366-0.0537) (\#C(\FF_{q^2}) - \#C(\FF_q))\\
&\qquad + 0.1137 \#C(\FF_{q^3}) + 0.0537 \#C(\FF_{q^4}) \\
&= a_1 + 0.3366 (2a_2) + 0.1137 (3a_3) + 0.0537 (4a_4) \leq 0.8042g + 5.619.
\end{align*}
This yields the claimed results.
\end{proof}

\begin{lemma} \label{lem:degree 2 bound}
Suppose that $q=2$ and $g>1$.
\begin{enumerate}
\item[(a)]
If $d=2$, then
\[
(g,g') \in \{(2,3), (2,4),(2,5), (3,5),(3,6), (4,7), (4,8), (5,9), (6, 11), (7,13)\}.
\]
\item[(b)]
If $d=3$, then $(g,g') \in \{(2,4), (2,6), (3,7), (4, 10)\}$.
\item[(c)]
If $d=4$, then $(g,g') \in \{(2,5), (2,6), (3,9)\}$.
\item[(d)]
If $d > 4$, then $g=2$ and $(d, g') \in \{(5, 6), (6, 7), (7, 8)\}$.
\end{enumerate}
\end{lemma}
\begin{proof}
We run an exhaustive search over Weil polynomials as in \S\ref{sec:exhaustion}, but also accounting for 
\eqref{eq:unique ramification point},
\eqref{eq:unique ramification point2},
\eqref{eq:cyclic} (for $d=2$),
\eqref{eq:double cover without trace} (taking $i=1,2,3$),
\eqref{eq:triple cover consequence} (taking $i=1,2$),
\eqref{eq:quadruple cover consequence} (taking $i=1$),
and \eqref{eq:cor Delta sequence}.
This rules out 
\begin{gather*}
(d,g,g') \in \{ (2,2,6), (2,2,7), (2,3,7), (2,3,8), (2,3,9), (2,4,9), (2,4,10), (2,4,11), \\
(2,5,10), (2,5,11), (2,6,12), (2,6,13), (2,7,14), (2,8,15), (2,9,17), \\
(3,2,5), (3,2,7), (3,2,8), (3,3,8), (3,3,9), (3,3,10), (3,4,11), (3,4,12), \\
(3,5,13), (3,5,14), (3,6,16), 
(4,2,7), (4,2,8), (4,3,10), (5,2,7), (5,2,8), (6,2,8)\}.
\end{gather*}
(The runtime is dominated by the cases $(d,g,g') = (2,8,15), (2,9,17)$.)
We may thus deduce (a) from \eqref{eq:riemann-hurwitz1} and Lemma~\ref{lem:degree 2 bound1}, (b) from Lemma~\ref{cor:genus bound q2} and Lemma~\ref{lem:genus bound q2 refinement}, and (c) and (d) from Lemma~\ref{lem:genus bound q2 refinement}.
\end{proof}

We obtain Theorem~\ref{T:purely geometric bounds2}(b) by a similar calculation which also accounts for \eqref{eq:double cover parity} (taking $j = 2$),
Remark~\ref{R:relative quadratic twist}, and the following Remark~\ref{R:resultant degree 2}.

\begin{remark} \label{R:resultant degree 2}
If $C' \to C$ is \'etale and geometrically cyclic (i.e., cyclic after base extension from $\FF_2$ to an algebraic closure), we can upgrade Lemma~\ref{lem:Delta sequence} to say that $\Delta$ has exponent exactly $d$ (because $\ker(f^*)$ is \'etale and cyclic of order $d$; compare \eqref{eq:cyclic}), and Corollary~\ref{cor:Delta sequence} to say that $\widetilde{\res}(h_1,h_2)$ must be divisible by $d$.

If we drop these conditions on $C' \to C$, we can still say something when
$\gcd(d, \widetilde{\res}(h_1,h_2)) = 2$: as in \cite[Theorem~2.2]{howe-lauter2} there must be a degree-2 map from $C'$ to another curve $D$ whose Jacobian is isogenous to $J(C)$ or $A$. By \eqref{eq:riemann-hurwitz1}, the second option cannot occur if $g' > 2g+1$;
in characteristic 2, \eqref{eq:deuring-shafarevich} also applies.

In the context of Theorem~\ref{T:purely geometric bounds2}(b), the condition that  $\gcd(d, \widetilde{\res}(h_1,h_2)) = 2$ rules out some cases
with $(d,g,g') \in \{(4,2,6), (4,3,9), (6,2,7)\}$: there would have to be a double cover $C' \to D$ with $J(D)$ isogenous to $J(C)$, but this is forbidden by Lemma~\ref{lem:degree 2 bound}(a).
Similarly, if $(d,g,g') = (4,2,5)$, then $J(D)$ cannot be isogenous to $A$: otherwise $D$ would admit an \'etale double cover while $\#J(D)(\FF_2) = 1$. Hence $J(C), J(C')$ must occur in Theorem~\ref{T:purely geometric bounds2}(c) with $(d,g,g') = (2,2,5)$.
\end{remark}

\begin{remark} \label{R:principal polarization}
When $d=2$ and $\delta \leq 1$, $A$ admits a principal polarization;
over $\CC$ this is classical \cite[Theorem~12.3.3]{birkenhake-lange}, and a characteristic-free argument will appear in \cite{achter-casalaina-martin}.
Our formulation of Theorem~\ref{T:purely geometric bounds2}(b) does not account for this constraint;
it would rule out a further 16 pairs, which are marked with stars in
Table~\ref{table:geometric bounds}.
\end{remark}

To obtain Theorem~\ref{T:purely geometric bounds2}(c), we use table lookups to find candidates for $C$ with a given Weil polynomial (see \S\ref{sec:exhaustion}), then use \Magma{} to enumerate cyclic extensions. As a consistency check, for each triple $(d,g,g')$ listed in Lemma~\ref{lem:degree 2 bound} with $g \leq 5$,
we enumerated cyclic extensions for \emph{all} curves $C$ of genus $g$; this took about 14 hours and yielded no new results.

\clearpage
\appendix

\section{Extensions of relative class number 1}
\label{sec:tables}

\begin{table}[ht]
\tiny
\begin{minipage}[t]{5.5cm}
\begin{tabular}{c|c|c|c|c|c}
$q_F$ & $g_F$ & $g_{F'}$ & $J(C)$ & $J(C')$ & $\#C'$\\
\hline
$2$ & $0$ & $1$ & 0 & \avlink{1.2.ac} & 1 \\
\hline
$2$ & $0$ & $2$ & 0 & \avlink{2.2.ad\_f} & 1 \\
$2$ & $0$ & $2$ & 0 & \avlink{2.2.ac\_c} & 1 \\
\hline
$2$ & $0$ & $3$ & 0 & \avlink{3.2.ad\_c\_b} & 1\\
$2$ & $0$ & $3$ & 0 & \avlink{3.2.ad\_d\_ac} & 1 \\
\hline
$2$ & $0$ & $4$ & 0 & \avlink{4.2.ad\_c\_a\_b} & 1 \\
\hline
$2$ & $1$ & $2$ & \avlink{1.2.a} & \avlink{2.2.ac\_e} & 1\\
$2$ & $1$ & $2$ & \avlink{1.2.b} & \avlink{2.2.ab\_c} &1\\
$2$ & $1$ & $2$ & \avlink{1.2.c} & \avlink{2.2.a\_a} &1 \\
\hline
$2$ & $1$ & $3$ & \avlink{1.2.ac} & \avlink{3.2.ad\_d\_ac} & 1\\
$2$ & $1$ & $3$ & \avlink{1.2.ab} & \avlink{3.2.ad\_g\_ak} & 1\\
$2$ & $1$ & $3$ & \avlink{1.2.ab} & \avlink{3.2.ac\_c\_ad} & 1\\
$2$ & $1$ & $3$ & \avlink{1.2.b} & \avlink{3.2.ad\_g\_ai} & 1\\
$2$ & $1$ & $3$ & \avlink{1.2.b} & \avlink{3.2.ac\_e\_ah} & 1\\
\hline
$2$ & $1$ & $4$ & \avlink{1.2.a} & \avlink{4.2.ad\_e\_af\_i} & 1\\
$2$ & $1$ & $4$ & \avlink{1.2.a} & \avlink{4.2.ad\_f\_ai\_m} & 2\\
\end{tabular}
\end{minipage}
\begin{minipage}[t]{5.5cm}
\begin{tabular}{c|c|c|c|c|c}
$q_F$ & $g_F$ & $g_{F'}$ & $J(C)$ & $J(C')$ & $\#C'$ \\
\hline
$2$ & $1$ & $4$ & \avlink{1.2.c} & \avlink{4.2.ad\_f\_ag\_i} &1\\
$2$ & $1$ & $4$ & \avlink{1.2.c} & \avlink{4.2.ac\_c\_ae\_i} &2\\
\hline
$2$ & $1$ & $5$ & \avlink{1.2.b} & \avlink{5.2.ad\_c\_d\_ag\_h} & 3\\
$2$ & $1$ & $5$ & \avlink{1.2.b} & \avlink{5.2.ad\_c\_e\_ai\_i} & 3\\
$2$ & $1$ & $5$ & \avlink{1.2.b} & \avlink{5.2.ad\_e\_ag\_k\_ao} & 3\\
\hline
$2$ & $1$ & $6$ & \avlink{1.2.c} & \avlink{6.2.ad\_c\_a\_f\_am\_q} & 1$^*$ \\
\hline
\hline
$3$ & $0$ & $1$ & 0 & \avlink{1.3.ab} & 1\\
\hline
$3$ & $1$ & $2$ & \avlink{1.3.ab} & \avlink{2.3.ae\_j} & 1\\
$3$ & $1$ & $2$ & \avlink{1.3.a} & \avlink{2.3.ad\_g} & 1\\
$3$ & $1$ & $2$ & \avlink{1.3.b} & \avlink{2.3.ac\_d} & 2\\
$3$ & $1$ & $2$ & \avlink{1.3.c} & \avlink{2.3.ab\_a} & 1\\
\hline
$3$ & $1$ & $3$ & \avlink{1.3.c} & \avlink{3.3.ae\_g\_ag} & 1\\
$3$ & $1$ & $3$ & \avlink{1.3.d} & \avlink{3.3.ad\_a\_j} & 2\\
\hline
\hline
$4$ & $0$ & $1$ & 0 & \avlink{1.4.ae} & 1\\
\hline
$4$ & $1$ & $2$ & \avlink{1.4.a} & \avlink{2.4.ae\_i} & 1\\
\end{tabular}
\end{minipage}
\medskip
\caption{Purely geometric extensions with $g_F \leq 1$, $g_{F'} > g_F$. The column $\#C'$ counts Jacobians in the isogeny class. The star indicates a conjectural value; see  Remark~\ref{R:unique curve}.}
\label{table:geometric extensions big q1}
\end{table}

\begin{table}[ht]
\tiny
\begin{tabular}{c|c|c|c|c|l}
$q_F$ & $d$ & $g_F$ & $g_{F'}$ & $J(C)$ & $F$ \\
\hline
$3$ & $2$ & $2$ & $3$ & $\avlink{2.3.ab\_c}$ & $y^{2} + x^{5} + 2 x^{2} + x$ \\
$3$ & $2$ & $2$ & $3$ & $\avlink{2.3.ab\_e}$ & $y^{2} + x^{6} + x^{4} + 2 x^{3} + x^{2} + 2 x$ \\
$3$ & $2$ & $2$ & $3$ & $\avlink{2.3.b\_c}$ & $y^{2} + 2 x^{5} + x^{2} + 2 x$ \\
$3$ & $2$ & $2$ & $3$ & $\avlink{2.3.b\_e}$ & $y^{2} + 2 x^{6} + 2 x^{4} + x^{3} + 2 x^{2} + x$ \\
\hline
$3$ & $2$ & $2$ & $4$ & $\avlink{2.3.c\_h}$ & $y^{2} + 2 x^{6} + x^{4} + 2 x^{3} + x^{2} + 2$ \\
\hline
$3$ & $2$ & $3$ & $5$ & $\avlink{3.3.c\_g\_i}$ & $y^{2} + 2 x^{8} + x^{7} + x^{5} + x^{3} + 2 x^{2} + 2 x$ \\
$3$ & $2$ & $3$ & $5$ & $\avlink{3.3.c\_g\_m}$ & $y^{2} + x^{7} + 2 x^{5} + x^{4} + x^{3} + x^{2} + 2$ \\
\hline
$3$ & $3$ & $2$ & $4$ & $\avlink{2.3.c\_d}$ & $y^{2} + 2 x^{6} + 2 x^{4} + x^{3} + x + 2$ \\
$3$ & $3$ & $2$ & $4$ & $\avlink{2.3.c\_g}$ & $y^{2} + 2 x^{6} + 2 x^{5} + x^{4} + x^{3} + x^{2} + 2 x + 2$ \\
\hline
\hline
$4$ & $2$ & $2$ & $3$ & $\avlink{2.4.ab\_e}$ & $y^{2} + x y + x^{5} + x$ \\
$4$ & $2$ & $2$ & $3$ & $\avlink{2.4.b\_e}$ & $y^{2} + x y + x^{5} + a x^{2} + x$ \\
\hline
$4$ & $3$ & $2$ & $4$ & $\avlink{2.4.d\_h}$ & $y^{2} + (x^{3} + x + 1) y + a x^{5} + a x^{4} + a x^{3} + a x$ \\
\end{tabular}
\medskip

\caption{Purely geometric extensions with $q_F > 2$ and $g_F > 1$.}
\label{table:geometric extensions big q2}
\end{table}

\begin{table}[ht]
\tiny
\begin{tabular}{c|c|c|c|l}
$d$ & $g_F$ & $g_{F'}$ & $J(C)$ & $F$ \\
\hline
$3$ & $2$ & $4$ & $\avlink{2.2.ac\_e}$ & $y^{2} + y + x^{5} + x^{4} + 1$ \\
$3$ & $2$ & $4$ & $\avlink{2.2.b\_b}$ & $y^{2} + (x^{3} + x + 1) y + x^{6} + x^{3} + x^{2} + x$ \\
\hline
$3$ & $2$ & $6$ & $\avlink{2.2.a\_c}$ & $y^{2} + y + x^{5} + x^{4} + x^{3}$ \\
$3$ & $2$ & $6$ & $\avlink{2.2.b\_c}$ & $y^{2} + x y + x^{5} + x^{3} + x^{2} + x$ \\
$3$ & $2$ & $6$ & $\avlink{2.2.b\_d}$ & $y^{2} + (x^{3} + x + 1) y + x^{6} + x^{5} + x^{4} + x^{2}$ \\
\hline
$3$ & $3$ & $7$ & $\avlink{3.2.a\_b\_a}$ & $y^{4} + (x^{3} + 1) y + x^{4}$ \\
$3$ & $3$ & $7$ & $\avlink{3.2.a\_b\_d}$ & $y^{3} + x^{2} y^{2} + x^{3} y + x^{4} + x^{3} + x$ \\
$3$ & $3$ & $7$ & $\avlink{3.2.b\_b\_b}$ & $y^{3} + x y^{2} + (x^{3} + 1) y + x^{4}$ \\
$3$ & $3$ & $7$ & $\avlink{3.2.b\_b\_e}$ & $y^{3} + (x^{2} + x) y^{2} + y + x^{3}$ \\
$3$ & $3$ & $7$ & $\avlink{3.2.b\_c\_b}$ & $x y^{3} + x y^{2} + y + x^{3}$ \\
$3$ & $3$ & $7$ & $\avlink{3.2.b\_c\_e}$ & $y^{4} + x y^{2} + y + x^{4}$ \\
$3$ & $3$ & $7$ & $\avlink{3.2.b\_e\_e}$ & $y^{3} + x^{2} y^{2} + x y + x^{4} + x$ \\
\hline
$3$ & $4$ & $10$ & $\avlink{4.2.d\_f\_k\_s}$ & $x^{2} y^{3} + (x^{4} + x^{2} + 1) y + x^{4} + x^{2} + x + 1$ \\
$3$ & $4$ & $10$ & $\avlink{4.2.e\_j\_q\_z}$ & $x y^{3} + (x^{2} + x + 1) y^{2} + (x^{4} + x) y + x^{5} + x^{4}$ \\
\hline
$4$ & $2$ & $5$ & $\avlink{2.2.ab\_c}$ & $y^{2} + x y + x^{5} + x^{3} + x$ \\
\hline
$5$ & $2$ & $6$ & $\avlink{2.2.a\_a}$ & $y^{2} + y + x^{5}$ \\
$5$ & $2$ & $6$ & $\avlink{2.2.b\_c}$ & $y^{2} + x y + x^{5} + x^{3} + x^{2} + x$ \\
$5$ & $2$ & $6$ & $\avlink{2.2.c\_e}$ & $y^{2} + y + x^{5} + x^{4}$ \\
\hline
$7$ & $2$ & $8$ & $\avlink{2.2.c\_d}$ & $y^{2} + (x^{2} + x + 1) y + x^{5} + x^{4} + x^{2} + x$ \\
\end{tabular}\medskip
\caption{Cyclic purely geometric extensions with $q_F = 2$, $g_F > 1$, $d>2$. Conjecture~\ref{conj:remaining problem} asserts that no noncyclic extensions occur.}
\label{table:geometric extensions big q2b1}
\end{table}

\begin{table}[ht]
\tiny
\begin{tabular}{c|c|c|p{8cm}}
$g_F$ & $g_{F'}$ & $J(C)$ & $F$ \\
\hline
$2$ & $3$ & $\avlink{2.2.ab\_c}$ & $y^{2} + x y + x^{5} + x^{3} + x$ \\
$2$ & $3$ & $\avlink{2.2.b\_c}$ & $y^{2} + x y + x^{5} + x^{3} + x^{2} + x$ \\
\hline
$2$ & $4$ & $\avlink{2.2.a\_a}$ & $y^{2} + y + x^{5}$ \\
$2$ & $4$ & $\avlink{2.2.a\_c}$ & $y^{2} + y + x^{5} + x^{4} + x^{3}$ \\
$2$ & $4$ & $\avlink{2.2.b\_b}$ & $y^{2} + (x^{3} + x + 1) y + x^{6} + x^{3} + x^{2} + x$ \\
\hline
$2$ & $5$ & $\avlink{2.2.b\_d}$ & $y^{2} + (x^{3} + x + 1) y + x^{6} + x^{5} + x^{4} + x^{2}$ \\
$2$ & $5$ & $\avlink{2.2.c\_e}$ & $y^{2} + y + x^{5} + x^{4}$ \\
\hline
$3$ & $5$ & $\avlink{3.2.ad\_g\_ai}$ & $y^{2} + (x^{4} + x^{2} + 1) y + x^{8} + x + 1$ \\
$3$ & $5$ & $\avlink{3.2.ab\_a\_c}$ & $x y^{3} + (x^{2} + x) y^{2} + y + x^{4}$ \\
$3$ & $5$ & $\avlink{3.2.ab\_a\_c}$ & $x y^{3} + x^{2} y^{2} + (x^{2} + 1) y + x^{4}$ \\
$3$ & $5$ & $\avlink{3.2.ab\_c\_ac}$ & $y^{2} + x y + x^{7} + x^{5} + x$ \\
$3$ & $5$ & $\avlink{3.2.a\_a\_f}$ & $x y^{3} + y + x^{3}$ \\
$3$ & $5$ & $\avlink{3.2.a\_c\_ab}$ & $y^{2} + (x^{4} + x^{2} + x + 1) y + x^{6} + x^{5} + x^{2} + 1$ \\
$3$ & $5$ & $\avlink{3.2.a\_c\_b}$ & $y^{2} + (x^{4} + x^{2} + x + 1) y + x^{8} + x^{6} + x^{5} + x^{4}$ \\
$3$ & $5$ & $\avlink{3.2.b\_c\_c}$ & $y^{2} + x y + x^{7} + x^{5} + x^{2} + x$ \\
$3$ & $5$ & $\avlink{3.2.b\_c\_e}$ & $y^{2} + (x^{4} + x^{2}) y + x^{2} + x$ \\
\hline
$3$ & $6$ & $\avlink{3.2.b\_d\_c}$ & $y^{3} + x^{2} y^{2} + x^{2} y + x^{4} + x^{3} + x^{2} + x$ \\
$3$ & $6$ & $\avlink{3.2.b\_d\_e}$ & $x y^{3} + (x + 1) y^{2} + x^{4} + x^{3} + x$ \\
$3$ & $6$ & $\avlink{3.2.b\_e\_d}$ & $y^{3} + x^{2} y^{2} + (x^{3} + x^{2}) y + x^{4} + x$ \\
$3$ & $6$ & $\avlink{3.2.c\_d\_d}$ & $(x + 1) y^{3} + y + x^{3}$ \\
\hline
$4$ & $7$ & $\avlink{4.2.a\_c\_ab\_c}$ & $x^{3} y^{3} + (x^{3} + x^{2}) y + x^{6} + x^{3} + 1$ \\
$4$ & $7$ & $\avlink{4.2.a\_c\_ab\_g}$ & $(x^{2} + 1) y^{4} + (x^{3} + x^{2} + x + 1) y^{3} + (x^{5} + x^{4}) y + x^{6} + x^{3} + x^{2}$ \\
$4$ & $7$ & $\avlink{4.2.a\_c\_b\_c}$ & $x^{2} y^{4} + (x^{3} + 1) y^{2} + (x^{3} + x^{2} + x + 1) y + x^{6} + x^{5} + x^{3} + x^{2}$ \\
$4$ & $7$ & $\avlink{4.2.a\_c\_d\_c}$ & $(x^{2} + x + 1) y^{4} + (x^{3} + x^{2}) y^{3} + (x^{4} + x^{3} + 1) y^{2} + (x^{4} + x^{3} + x^{2}) y + x^{5} + x^{4} + x^{3} + x$ \\
$4$ & $7$ & $\avlink{4.2.a\_d\_b\_f}$ & $x^{3} y^{3} + (x^{3} + x^{2}) y + x^{6} + x^{5} + 1$ \\
$4$ & $7$ & $\avlink{4.2.a\_d\_b\_h}$ & $(x^{2} + x + 1) y^{4} + x^{3} y^{3} + (x^{4} + x^{2} + 1) y^{2} + x^{5} + x^{3} + x$ \\
$4$ & $7$ & $\avlink{4.2.b\_b\_c\_f}$ & $(x + 1) y^{3} + (x^{2} + x) y^{2} + (x^{3} + x) y + x^{5}$ \\
$4$ & $7$ & $\avlink{4.2.b\_c\_a\_a}$ & $y^{2} + x^{2} y + x^{9} + x^{7} + x + 1$ \\
$4$ & $7$ & $\avlink{4.2.c\_e\_h\_k}$ & $y^{2} + (x^{3} + x + 1) y + x^{9} + x^{7}$ \\
\hline
$4$ & $8$ & $\avlink{4.2.d\_i\_o\_x}$ & $(x + 1) y^{3} + (x^{3} + x^{2} + 1) y^{2} + x y + x^{4}$ \\
\hline
$5$ & $9$ & $\avlink{5.2.ab\_d\_b\_b\_j}$ & $(x^{4} + x^{3} + x^{2}) y^{4} + (x^{5} + x^{3} + x) y^{3} + (x^{3} + 1) y^{2} + (x^{7} + x + 1) y + x^{7} + x^{4} + x + 1$ \\
$5$ & $9$ & $\avlink{5.2.b\_c\_e\_i\_i}$ & $y^{2} + x^{3} y + x^{11} + x^{9} + x^{5} + x^{3} + x^{2} + x$ \\
$5$ & $9$ & $\avlink{5.2.b\_c\_e\_i\_i}$ & $y^{4} + (x^{4} + x^{2}) y^{2} + (x^{4} + x^{2} + 1) y + x^{8} + x^{6} + x^{4} + x$ \\
$5$ & $9$ & $\avlink{5.2.b\_f\_f\_p\_l}$ & $(x^{4} + x^{3} + x^{2}) y^{4} + (x^{5} + x^{3} + x^{2}) y^{3} + (x^{6} + x^{3} + x^{2} + x + 1) y^{2} + (x^{7} + x^{5} + x^{4} + x^{3} + 1) y + x^{5} + x^{4} + x^{3} + x^{2} + x + 1$ \\
$5$ & $9$ & $\avlink{5.2.b\_f\_f\_p\_p}$ & $y^{4} + x^{2} y^{3} + (x^{4} + x^{3} + x) y^{2} + (x^{5} + 1) y + x^{3} + x^{2} + x + 1$ \\
$5$ & $9$ & $\avlink{5.2.c\_e\_f\_k\_o}$ & $y^{2} + (x^{3} + x + 1) y + x^{12} + x^{11} + x^{10} + x^{7} + x^{5} + x^{3}$ \\
$5$ & $9$ & $\avlink{5.2.c\_f\_i\_n\_r}$ & $y^{4} + (x^{2} + x) y^{3} + (x^{4} + x^{3} + x^{2} + 1) y^{2} + (x^{6} + x^{5} + x^{4} + 1) y + x^{7} + x^{6} + x + 1$ \\
$5$ & $9$ & $\avlink{5.2.c\_f\_i\_p\_t}$ & $(x^{4} + x^{2} + x) y^{4} + (x^{4} + x^{3} + x + 1) y^{3} + (x^{6} + x^{2}) y^{2} + (x^{6} + x^{3} + x^{2} + x) y + x^{6} + x^{5} + 1$ \\
$5$ & $9$ & $\avlink{5.2.c\_f\_i\_p\_v}$ & $(x^{2} + x + 1) y^{6} + x y^{5} + (x^{4} + x) y^{4} + (x^{5} + x^{4} + x^{3} + x^{2} + x + 1) y^{3} + (x^{5} + x^{3} + 1) y^{2} + (x^{6} + x^{4} + x^{2}) y + x^{8} + x^{7} + x^{6} + x^{5} + x^{4} + x^{3} + x^{2} + x$ \\
$5$ & $9$ & $\avlink{5.2.c\_g\_j\_q\_u}$ & $y^{4} + y^{3} + (x^{4} + x^{3} + x^{2}) y^{2} + (x^{3} + x^{2} + 1) y + x^{6} + 1$ \\
$5$ & $9$ & $\avlink{5.2.d\_h\_n\_z\_bl}$ & $y^{2} + (x^{6} + x^{5} + x^{4} + x^{3} + x^{2} + x + 1) y + x^{10} + x^{6} + x^{4} + x^{3}$ \\
$5$ & $9$ & $\avlink{5.2.d\_i\_q\_bc\_bs}$ & $y^{4} + (x^{4} + x^{2}) y^{2} + (x^{4} + x^{2} + 1) y + x^{6} + x^{5}$ \\
\hline
\hline
$6$ & $11$ & $\avlink{6.2.c\_h\_k\_z\_bd\_cg}$ & $x^{2} y^{5} + (x^{3} + x) y^{4} + x^{4} y^{3} + (x^{5} + x^{4} + x^{3} + x^{2} + x + 1) y^{2} + (x^{6} + x^{3} + x^{2}) y + x^{7} + x^{3} + x + 1$ \\
$6$ & $11$ & $\avlink{6.2.d\_j\_t\_bn\_cl\_du}$ & $(x^{2} + x + 1) y^{4} + (x^{3} + x + 1) y^{3} + (x^{4} + x^{2} + 1) y^{2} + (x^{5} + x^{4} + 1) y + x^{5} + x^{4} + x^{3} + x$ \\
\hline
$7$ & $13$ & $(6, 18, 12, 18, 6, 60, 174)$ & $y^{4} + (x^{6} + x^{4} + x^{3} + x^{2} + 1) y^{2} + (x^{6} + x^{4} + x^{3} + x^{2}) y + x^{10} + x^{9} + x^{7} + x^{6}$ \\
\end{tabular}
\medskip
\caption{Purely geometric extensions with $q_F = 2$, $g_F > 1$, $d=2$. Completeness of the list is confirmed above the double line and conjectural below it
(Conjecture~\ref{conj:remaining problem}). For $g_F= 7$, $J(C)$ does not appear in LMFDB, so we list $J(C)(\FF_{2^i})$ for $i=1,\dots,7$.}
\label{table:geometric extensions big q2b}
\end{table}

\begin{table}
\tiny
\begin{tabular}{c|c|c}
$(g,g')$ & $A$ & $J(C)$ \\
\hline
$(2, 3)$ & $\avlink{1.2.ac}$ & $\avlink{2.2.ab\_c},\avlink{2.2.b\_c}$ \\
\hline
$(2, 4)$ & $\avlink{2.2.ab\_ab}\star$ & $\avlink{2.2.ab\_d}$ \\
$(2, 4)$ & $\avlink{2.2.ac\_c}$ & $\avlink{2.2.a\_a},\avlink{2.2.a\_c}$ \\
$(2, 4)$ & $\avlink{2.2.ad\_f}$ & $\avlink{2.2.b\_b}$ \\
\hline
$(2, 5)$ & $\avlink{3.2.ad\_d\_ac}$ & $\avlink{2.2.b\_d}$ \\
$(2, 5)$ & $\avlink{3.2.ae\_i\_am}$ & $\avlink{2.2.c\_e}$ \\
\hline
$(3, 5)$ & $\avlink{2.2.a\_ae}$ & $\avlink{3.2.ad\_g\_ai}$ \\
$(3, 5)$ & $\avlink{2.2.ac\_c}$ & $\avlink{3.2.ab\_a\_c},\avlink{3.2.ab\_c\_ac},\avlink{3.2.b\_c\_c}$ \\
$(3, 5)$ & $\avlink{2.2.ad\_f}$ & $\avlink{3.2.a\_a\_f},\avlink{3.2.a\_c\_ab},\avlink{3.2.a\_c\_b},\avlink{3.2.c\_e\_f}$ \\
$(3, 5)$ & $\avlink{2.2.ae\_i}$ & $\avlink{3.2.b\_c\_e}$ \\
\hline
$(3, 6)$ & $\avlink{3.2.ad\_c\_b}$ & $\avlink{3.2.b\_e\_d}$ \\
$(3, 6)$ & $\avlink{3.2.ad\_d\_ac}$ & $\avlink{3.2.b\_d\_c},\avlink{3.2.b\_d\_e}$ \\
$(3, 6)$ & $\avlink{3.2.ae\_i\_am}$ & $\avlink{3.2.c\_e\_e},\avlink{3.2.c\_e\_g}$ \\
$(3, 6)$ & $\avlink{3.2.ae\_j\_ap}$ & $\avlink{3.2.c\_d\_d}$ \\
\hline
$(4, 7)$ & $\avlink{3.2.ad\_c\_b}$ & $\avlink{4.2.a\_d\_ab\_f},\avlink{4.2.a\_d\_ab\_h},\avlink{4.2.a\_d\_b\_f},\avlink{4.2.a\_d\_b\_h},\avlink{4.2.a\_d\_d\_f}$ \\
$(4, 7)$ & $\avlink{3.2.ad\_d\_ac}$ & $\avlink{4.2.a\_c\_ab\_c},\avlink{4.2.a\_c\_ab\_e},\avlink{4.2.a\_c\_ab\_g},\avlink{4.2.a\_c\_b\_c},\avlink{4.2.a\_c\_b\_e}$ \\
$(4, 7)$ & $\avlink{3.2.ad\_d\_ac}$ & $\avlink{4.2.a\_c\_b\_g},\avlink{4.2.a\_c\_d\_a},\avlink{4.2.a\_c\_d\_c},\avlink{4.2.a\_e\_b\_k},\avlink{4.2.c\_g\_j\_q}$ \\
$(4, 7)$ & $\avlink{3.2.ae\_i\_am}$ & $\avlink{4.2.b\_c\_a\_a},\avlink{4.2.b\_c\_a\_c},\avlink{4.2.b\_c\_a\_e},\avlink{4.2.b\_c\_c\_c}$ \\
$(4, 7)$ & $\avlink{3.2.ae\_i\_am}$ & $\avlink{4.2.b\_c\_c\_e},\avlink{4.2.b\_c\_c\_g},\avlink{4.2.b\_c\_e\_e},\avlink{4.2.b\_e\_c\_i},\avlink{4.2.b\_e\_e\_k}$ \\
$(4, 7)$ & $\avlink{3.2.ae\_j\_ap}$ & $\avlink{4.2.b\_b\_a\_b},\avlink{4.2.b\_b\_a\_d},\avlink{4.2.b\_b\_c\_d},\avlink{4.2.b\_b\_c\_f},\avlink{4.2.b\_b\_c\_h},\avlink{4.2.b\_d\_c\_h},\avlink{4.2.b\_d\_e\_j}$ \\
$(4, 7)$ & $\avlink{3.2.af\_n\_aw}$ & $\avlink{4.2.c\_e\_h\_k}$ \\
\hline
$(4, 8)$ & $\avlink{4.2.ae\_g\_ae\_c}$ & $\avlink{4.2.c\_g\_i\_q}$ \\
$(4, 8)$ & $\avlink{4.2.af\_m\_au\_bd}$ & $\avlink{4.2.d\_i\_o\_x}$ \\
$(4, 8)$ & $\avlink{4.2.af\_n\_az\_bn}\star$ & $\avlink{4.2.d\_h\_l\_r}$ \\
\hline
$(5, 9)$ & $\avlink{4.2.ac\_ab\_ac\_n}$ & $\avlink{5.2.ab\_d\_b\_b\_j}$ \\
$(5, 9)$ & $\avlink{4.2.ad\_c\_a\_b}$ & $\avlink{5.2.a\_d\_c\_j\_d}$ \\
$(5, 9)$ & $\avlink{4.2.ad\_d\_ag\_o}$ & $\avlink{5.2.a\_c\_d\_e\_g}$ \\
$(5, 9)$ & $\avlink{4.2.ae\_f\_c\_al}$ & $\avlink{5.2.b\_f\_f\_p\_l},\avlink{5.2.b\_f\_f\_p\_n},\avlink{5.2.b\_f\_f\_p\_p}$ \\
$(5, 9)$ & $\avlink{4.2.ae\_g\_ae\_c}$ & $\avlink{5.2.b\_e\_c\_i\_a},\avlink{5.2.b\_e\_c\_i\_c},\avlink{5.2.b\_e\_c\_k\_e},\avlink{5.2.b\_e\_e\_k\_i}$ \\
$(5, 9)$ & $\avlink{4.2.ae\_g\_ae\_c}$ & $\avlink{5.2.b\_e\_e\_k\_k},\avlink{5.2.b\_e\_e\_k\_m},\avlink{5.2.b\_e\_e\_m\_k},\avlink{5.2.b\_e\_e\_m\_m},\avlink{5.2.b\_e\_g\_m\_q}$ \\
$(5, 9)$ & $\avlink{4.2.ae\_h\_ak\_p}$ & $\avlink{5.2.b\_d\_d\_h\_d},\avlink{5.2.b\_d\_d\_h\_f},\avlink{5.2.b\_d\_d\_h\_h},\avlink{5.2.b\_d\_d\_h\_j}$ \\
$(5, 9)$ & $\avlink{4.2.ae\_h\_ak\_p}$ & $\avlink{5.2.b\_d\_d\_j\_h},\avlink{5.2.b\_d\_d\_j\_j},\avlink{5.2.b\_d\_d\_j\_l},\avlink{5.2.b\_d\_f\_j\_j},\avlink{5.2.b\_d\_f\_j\_l}$ \\
$(5, 9)$ & $\avlink{4.2.ae\_i\_aq\_bc}$ & $\avlink{5.2.b\_c\_e\_i\_i}$ \\
$(5, 9)$ & $\avlink{4.2.af\_l\_ao\_q}$ & $\avlink{5.2.c\_g\_j\_q\_u},\avlink{5.2.c\_g\_j\_s\_w},\avlink{5.2.c\_g\_j\_u\_y},\avlink{5.2.c\_g\_l\_u\_bc}$ \\
$(5, 9)$ & $\avlink{4.2.af\_m\_au\_bd}$ & $\avlink{5.2.c\_f\_g\_l\_l},\avlink{5.2.c\_f\_g\_n\_p},\avlink{5.2.c\_f\_i\_n\_r},\avlink{5.2.c\_f\_i\_n\_t},\avlink{5.2.c\_f\_i\_p\_t},\avlink{5.2.c\_f\_i\_p\_v},\avlink{5.2.c\_f\_k\_r\_z}$ \\
$(5, 9)$ & $\avlink{4.2.af\_n\_aba\_bq}$ & $\avlink{5.2.c\_e\_f\_k\_m},\avlink{5.2.c\_e\_f\_k\_o},\avlink{5.2.c\_e\_f\_m\_q}$ \\
$(5, 9)$ & $\avlink{4.2.af\_n\_az\_bn}\star$ & $\avlink{5.2.c\_e\_e\_g\_f},\avlink{5.2.c\_e\_e\_g\_h},\avlink{5.2.c\_e\_e\_i\_j},\avlink{5.2.c\_e\_e\_k\_n},\avlink{5.2.c\_e\_g\_k\_l},\avlink{5.2.c\_e\_g\_k\_n},\avlink{5.2.c\_e\_g\_k\_p}$ \\
$(5, 9)$ & $\avlink{4.2.af\_n\_az\_bn}\star$ & $\avlink{5.2.c\_e\_g\_k\_r},\avlink{5.2.c\_e\_g\_m\_p},\avlink{5.2.c\_e\_g\_m\_r},\avlink{5.2.c\_e\_g\_m\_t},\avlink{5.2.c\_e\_i\_o\_t},\avlink{5.2.c\_e\_i\_o\_v},\avlink{5.2.c\_g\_i\_s\_v}$ \\
$(5, 9)$ & $\avlink{4.2.ag\_s\_abk\_ce}$ & $\avlink{5.2.d\_i\_q\_bc\_bs}$ \\
$(5, 9)$ & $\avlink{4.2.ag\_t\_abp\_co}$ & $\avlink{5.2.d\_h\_o\_z\_bk},\avlink{5.2.d\_h\_o\_z\_bm}$ \\
$(5, 9)$ & $\avlink{4.2.ag\_t\_abq\_cr}$ & $\avlink{5.2.d\_h\_n\_z\_bl}$ \\
\hline
$(6, 11)$ & $\avlink{5.2.ae\_e\_a\_l\_abh}$ & $\avlink{6.2.b\_g\_i\_v\_ba\_bz}$ \\
$(6, 11)$ & $\avlink{5.2.ae\_f\_ae\_p\_abi}$ & $\avlink{6.2.b\_f\_h\_p\_t\_bk},\avlink{6.2.b\_f\_h\_p\_v\_bi},\avlink{6.2.b\_f\_h\_r\_v\_bq}$ \\
$(6, 11)$ & $\avlink{5.2.af\_k\_ak\_f\_ac}$ & $\avlink{6.2.c\_h\_k\_z\_bd\_cg}$ \\
$(6, 11)$ & $\avlink{5.2.af\_l\_as\_bg\_aca}$ & $\avlink{6.2.c\_g\_l\_w\_bg\_ca},\avlink{6.2.c\_g\_l\_w\_bg\_cc},\avlink{6.2.c\_g\_l\_w\_bg\_ce}$ \\
$(6, 11)$ & $\avlink{5.2.af\_l\_as\_bg\_aca}$ & $\avlink{6.2.c\_g\_l\_w\_bi\_ca},\avlink{6.2.c\_g\_l\_w\_bi\_cc},\avlink{6.2.c\_g\_l\_w\_bi\_ce}$ \\
$(6, 11)$ & $\avlink{5.2.af\_m\_aw\_bk\_acb}$ & $\avlink{6.2.c\_f\_i\_q\_v\_bh},\avlink{6.2.c\_f\_i\_q\_v\_bj}, \avlink{6.2.c\_f\_i\_q\_v\_bl}, \avlink{6.2.c\_f\_i\_q\_x\_bj}$ \\
$(6, 11)$ & $\avlink{5.2.af\_m\_aw\_bk\_acb}$ & $,\avlink{6.2.c\_f\_i\_q\_x\_bl}, \avlink{6.2.c\_f\_i\_q\_x\_bn},\avlink{6.2.c\_f\_i\_q\_x\_bp},\avlink{6.2.c\_f\_i\_q\_z\_bn}$ \\
$(6, 11)$ & $\avlink{5.2.af\_m\_aw\_bk\_acb}$ & $\avlink{6.2.c\_f\_i\_q\_z\_bp},\avlink{6.2.c\_f\_i\_s\_z\_bp},\avlink{6.2.c\_f\_i\_s\_z\_br}$ \\
$(6, 11)$ & $\avlink{5.2.ag\_r\_abg\_bx\_acs}$ & $\avlink{6.2.d\_j\_r\_bh\_bx\_cy},\avlink{6.2.d\_j\_r\_bh\_bx\_da},\avlink{6.2.d\_j\_r\_bh\_bz\_dc},\avlink{6.2.d\_j\_r\_bj\_cb\_di}$ \\
$(6, 11)$ & $\avlink{5.2.ag\_r\_abg\_bx\_acs}$ & $\avlink{6.2.d\_j\_r\_bj\_cd\_dm},\avlink{6.2.d\_j\_t\_bn\_cl\_ds},\avlink{6.2.d\_j\_t\_bn\_cl\_du},\avlink{6.2.d\_j\_t\_bn\_cl\_dw}$ \\
$(6, 11)$ & $\avlink{5.2.ag\_t\_abt\_di\_afe}$ & $\avlink{6.2.d\_h\_m\_x\_bi\_ca},\avlink{6.2.d\_h\_m\_x\_bk\_ce},\avlink{6.2.d\_h\_m\_x\_bm\_ci}$ \\
\hline
$(7, 13)$ & $\avlink{6.2.ag\_p\_aw\_bh\_acu\_ey}$ & $(6, 18, 12, 18, 6, 60, 174),(6, 18, 12, 18, 6, 72, 132),(6, 18, 12, 18, 6, 84, 90)$ \\
$(7, 13)$ & $\avlink{6.2.ah\_y\_ace\_ea\_agn\_jq}$ & $(7, 15, 7, 31, 12, 69, 126),(7, 15, 7, 31, 22, 45, 112)$\\
$(7, 13)$ & $\avlink{6.2.ah\_y\_ace\_ea\_agn\_jq}$ & $(7, 15, 7, 31, 22, 57, 70),(7, 15, 7, 31, 22, 57, 84)$ \\
\end{tabular}
\medskip
\begin{minipage}[t]{6.25cm}
\begin{tabular}{c|c|c}
$(g,g')$ & $A$ & $J(C)$ \\
\hline
$(2, 4)$ & $\avlink{2.2.ab\_ab}$ & $\avlink{2.2.ac\_e}$ \\
$(2, 4)$ & $\avlink{2.2.ad\_f}$ & $\avlink{2.2.a\_ab},\avlink{2.2.a\_c}$ \\
$(2, 4)$ & $\avlink{2.2.ae\_i}$ & $\avlink{2.2.b\_b},\avlink{2.2.c\_c}$ \\
\hline
$(2, 6)$ & $\avlink{4.2.ad\_b\_g\_am}$ & $\avlink{2.2.a\_c}$ \\
$(2, 6)$ & $\avlink{4.2.ae\_e\_h\_av}$ & $\avlink{2.2.b\_d}$ \\
$(2, 6)$ & $\avlink{4.2.ae\_f\_c\_al}$ & $\avlink{2.2.b\_c}$ \\
$(2, 6)$ & $\avlink{4.2.af\_l\_ao\_q}$ & $\avlink{2.2.c\_c}$ \\
\hline
$(3, 7)$ & $\avlink{4.2.ac\_ac\_e\_a}$ & $\avlink{3.2.ab\_c\_a}$ \\
$(3, 7)$ & $\avlink{4.2.ad\_b\_g\_am}$ & $\avlink{3.2.a\_b\_a},\avlink{3.2.a\_b\_d}$ \\
$(3, 7)$ & $\avlink{4.2.ae\_e\_h\_av}$ & $\avlink{3.2.b\_c\_b},\avlink{3.2.b\_c\_e}$ \\
$(3, 7)$ & $\avlink{4.2.ae\_e\_i\_ay}$ & $\avlink{3.2.b\_c\_a}$ \\
$(3, 7)$ & $\avlink{4.2.ae\_e\_i\_ay}$ & $\avlink{3.2.b\_c\_d}, \avlink{3.2.b\_d\_e}$ \\ 
$(3, 7)$ & $\avlink{4.2.ae\_e\_i\_ay}$ & $\avlink{3.2.c\_e\_h},\avlink{3.2.d\_h\_l}$ \\
$(3, 7)$ & $\avlink{4.2.ae\_f\_c\_al}$ & $\avlink{3.2.b\_b\_b}, \avlink{3.2.b\_b\_e}$ \\
$(3, 7)$ & $\avlink{4.2.ae\_f\_c\_al}$ & $\avlink{3.2.b\_c\_d}, \avlink{3.2.b\_e\_e}$ \\
$(3, 7)$ & $\avlink{4.2.ae\_f\_c\_al}$ & $\avlink{3.2.c\_d\_f},\avlink{3.2.c\_e\_h}$ \\
\hline
$(4, 10)$ & $\avlink{6.2.ag\_p\_ar\_ag\_cg\_aei}$ & $\avlink{4.2.d\_f\_i\_n}$ \\
$(4, 10)$ & $\avlink{6.2.ag\_p\_at\_g\_bb\_acj}$ & $\avlink{4.2.d\_f\_k\_s}$ \\
$(4, 10)$ & $\avlink{6.2.ah\_v\_abe\_a\_dk\_ahc}$ & $\avlink{4.2.e\_j\_q\_z},\avlink{4.2.e\_k\_u\_bg}$ \\
$(4, 10)$ & $\avlink{6.2.ai\_bc\_abw\_m\_ey\_alc}$ & $\avlink{4.2.f\_o\_bc\_bs}$ \\
\end{tabular}
\end{minipage}
\begin{minipage}[t]{6.25cm}
\begin{tabular}{c|c|c}
$(d,g,g')$ & $A$ & $J(C)$ \\
\hline
$(4, 2, 5)$ & $\avlink{3.2.ac\_ac\_i}$ & $\avlink{2.2.ab\_c}$ \\
$(4, 2, 5)$ & $\avlink{3.2.ae\_i\_am}$ & $\avlink{2.2.b\_a}$ \\
$(4, 2, 5)$ & $\avlink{3.2.ae\_i\_am}$ & $\avlink{2.2.b\_c},\avlink{2.2.c\_e}$ \\
\hline
$(4, 2, 6)$ & $\avlink{4.2.ae\_e\_i\_ay}$ & $\avlink{2.2.c\_e}$ \\
\hline
$(4, 3, 9)$ & $\avlink{6.2.af\_i\_ab\_ag\_an\_br}$ & $\avlink{3.2.c\_e\_f}$ \\
$(4, 3, 9)$ & $\avlink{6.2.ag\_o\_am\_am\_bw\_adc}$ & $\avlink{3.2.d\_g\_i}$ \\
\hline
$(5, 2, 6)$ & $\avlink{4.2.ad\_b\_g\_am}$ & $\avlink{2.2.a\_b}$ \\
$(5, 2, 6)$ & $\avlink{4.2.ad\_c\_a\_b}$ & $\avlink{2.2.a\_a}$ \\
$(5, 2, 6)$ & $\avlink{4.2.ae\_e\_h\_av}$ & $\avlink{2.2.b\_c}$ \\
$(5, 2, 6)$ & $\avlink{4.2.ae\_e\_i\_ay}$ & $\avlink{2.2.b\_d}$ \\
$(5, 2, 6)$ & $\avlink{4.2.ae\_h\_ak\_p}$ & $\avlink{2.2.b\_c}$ \\
$(5, 2, 6)$ & $\avlink{4.2.af\_l\_ao\_q}$ & $\avlink{2.2.c\_e},\avlink{2.2.d\_f}$ \\
$(5, 2, 6)$ & $\avlink{4.2.af\_n\_az\_bn}$ & $\avlink{2.2.c\_e}$ \\
\hline
$(6, 2, 7)$ & $\avlink{5.2.ae\_e\_e\_am\_q}$ & $\avlink{2.2.b\_c},\avlink{2.2.c\_e}$ \\
$(6, 2, 7)$ & $\avlink{5.2.af\_k\_ak\_f\_ac}$ & $\avlink{2.2.c\_c},\avlink{2.2.c\_d}$ \\
$(6, 2, 7)$ & $\avlink{5.2.af\_l\_as\_bg\_aca}$ & $\avlink{2.2.c\_c}$ \\
$(6, 2, 7)$ & $\avlink{5.2.ag\_q\_aba\_bh\_abr}$ & $\avlink{2.2.d\_f}$ \\
$(6, 2, 7)$ & $\avlink{5.2.ag\_r\_abg\_bx\_acs}$ & $\avlink{2.2.d\_f}$ \\
\hline
$(7, 2, 8)$ & $\avlink{6.2.af\_j\_ah\_d\_ab\_ab}$ & $\avlink{2.2.c\_d}$ \\
\end{tabular}
\end{minipage}

\medskip
\caption{Candidates for $A$ and $J(C)$ in Theorem~\ref{T:purely geometric bounds2}(b) for $d=2$, $d=3$, and $d>3$ respectively; $\star$ means $A$ is not principally polarizable (Remark~\ref{R:principal polarization}). For $g = 7$, $J(C)$ does not appear in LMFDB, so we list $J(C)(\FF_{2^i})$ for $i=1,\dots,7$.}
\label{table:geometric bounds}
\end{table}

\end{document}